\documentclass[%
  a4paper,
  onecolumn,
]{mypreprint}

\usepackage[english]{babel}

\usepackage{graphicx}

\usepackage{caption}
\usepackage{subcaption}

\usepackage{tikz}
\usepackage{pgfplots}
  \pgfplotsset{compat = 1.13}
  \pgfkeys{/pgf/number format/.cd,1000 sep={\,}}
  \usepgfplotslibrary{external}
  \tikzset{external/system call = {%
    pdflatex \tikzexternalcheckshellescape
      -halt-on-error
      -interaction=batchmode
      -jobname "\image" "\texsource"}}
  \tikzexternaldisable

\usepackage{amsmath}
\usepackage{amssymb}
\usepackage{amsthm}

\usepackage{dsfont}

\usepackage[linesnumbered, lined, algoruled]{algorithm2e}
\Crefname{algocf}{Algorithm}{Algorithms}

\theoremstyle{plain}\newtheorem{theorem}{Theorem}
\theoremstyle{plain}\newtheorem{lemma}{Lemma}
\theoremstyle{plain}\newtheorem{proposition}{Proposition}
\theoremstyle{plain}\newtheorem{corollary}{Corollary}
\theoremstyle{definition}\newtheorem{remark}{Remark}
\theoremstyle{definition}\newtheorem{assumption}{Assumption}

\newcommand{\trans}{\ensuremath{\mkern-1.5mu\mathsf{T}}}
\DeclareMathOperator{\diag}{diag}
\DeclareMathOperator{\real}{Re}
\DeclareMathOperator{\imag}{Im}
\DeclareMathOperator{\conv}{conv}

\newcommand{\R}{\ensuremath{\mathbb{R}}}
\newcommand{\C}{\ensuremath{\mathbb{C}}}

\newcommand{\bA}{\ensuremath{\boldsymbol{A}}}
\newcommand{\bE}{\ensuremath{\boldsymbol{E}}}
\newcommand{\bM}{\ensuremath{\boldsymbol{M}}}
\newcommand{\bD}{\ensuremath{\boldsymbol{D}}}
\newcommand{\bI}{\ensuremath{\boldsymbol{I}}}
\newcommand{\bK}{\ensuremath{\boldsymbol{K}}}
\newcommand{\bb}{\ensuremath{\boldsymbol{b}}}
\newcommand{\bc}{\ensuremath{\boldsymbol{c}}}
\newcommand{\be}{\ensuremath{\boldsymbol{e}}}
\newcommand{\bx}{\ensuremath{\boldsymbol{x}}}

\newcommand{\bLambda}{\ensuremath{\boldsymbol{\Lambda}}}
\newcommand{\bSigma}{\ensuremath{\boldsymbol{\Sigma}}}
\newcommand{\bTheta}{\ensuremath{\boldsymbol{\Theta}}}
\newcommand{\bones}{\boldsymbol{\mathds{1}}}
\newcommand{\bX}{\ensuremath{\boldsymbol{X}}}
\newcommand{\bu}{\ensuremath{\boldsymbol{u}}}
\newcommand{\bv}{\ensuremath{\boldsymbol{v}}}
\newcommand{\bz}{\ensuremath{\boldsymbol{z}}}
\newcommand{\bL}{\ensuremath{\boldsymbol{L}}}
\newcommand{\bg}{\ensuremath{\boldsymbol{g}}}
\newcommand{\bw}{\ensuremath{\boldsymbol{w}}}
\newcommand{\bfz}{\ensuremath{\boldsymbol{0}}}

\newcommand{\bhE}{\ensuremath{\skew4\widehat{\boldsymbol{E}}}}
\newcommand{\bhA}{\ensuremath{\skew4\widehat{\boldsymbol{A}}}}
\newcommand{\bhM}{\ensuremath{\skew4\widehat{\boldsymbol{M}}}}
\newcommand{\bhD}{\ensuremath{\skew4\widehat{\boldsymbol{D}}}}
\newcommand{\bhK}{\ensuremath{\skew4\widehat{\boldsymbol{K}}}}
\newcommand{\bhb}{\ensuremath{\widehat{\boldsymbol{b}}}}
\newcommand{\bhc}{\ensuremath{\widehat{\boldsymbol{c}}}}
\newcommand{\hSigma}{\ensuremath{\widehat{\Sigma}}}
\newcommand{\bhPhi}{\ensuremath{\widehat{\boldsymbol{\Phi}}}}
\newcommand{\bhPsi}{\ensuremath{\widehat{\boldsymbol{\Psi}}}}
\newcommand{\hH}{\ensuremath{\skew4\widehat{H}}}
\newcommand{\bhf}{\ensuremath{\skew4\widehat{\boldsymbol{f}}}}

\newcommand{\bcA}{\ensuremath{\boldsymbol{\mathcal{A}}}}
\newcommand{\bcE}{\ensuremath{\boldsymbol{\mathcal{E}}}}

\renewcommand{\i}{\ensuremath{\mathfrak{i}}}

\newcommand{\soloewnerk}{\mbox{\textsf{soBaryLoewK}}}
\newcommand{\soloewnerd}{\mbox{\textsf{soBaryLoewD}}}
\newcommand{\soloewnerkdo}{\mbox{\textsf{soBaryLoewKD0}}}
\newcommand{\soloewnerr}{\mbox{\textsf{soLoewRayleigh}}}
\newcommand{\loewner}{\mbox{\textsf{BaryLoew}}}

\newcommand{\relerr}{\ensuremath{\epsilon_{\mathsf{rel}}}}
\newcommand{\matlab}{\mbox{MATLAB}}

\usepackage{todonotes}

\definecolor{matlabBlue}{HTML}{0072BD}
\definecolor{matlabOrange}{HTML}{D95319}
\definecolor{matlabPurple}{HTML}{7E2F8E}
\definecolor{matlabGreen}{HTML}{77AC30}
\definecolor{matlabBrown}{HTML}{A2142F}

\pgfplotscreateplotcyclelist{plotlist}{
  {black, mark = *, only marks, mark size = 2.5pt},
  {matlabBlue, solid, line width = 1.75pt},
  {matlabPurple, dashed, line width = 1.75pt},
  {matlabGreen, dash dot, line width = 1.75pt},
  {matlabBrown, dotted, line width = 1.75pt}
}

\pgfplotscreateplotcyclelist{plotlist2}{
  {black, mark = *, only marks, mark size = 2.5pt},
  {matlabBlue, solid, line width = 1.75pt},
  {matlabGreen, dash dot, line width = 1.75pt},
  {matlabBrown, dotted, line width = 1.75pt}
}

\crefname{assumption}{Assumption}{Assumptions}

\begin{document}
  
%%%%%%%%%%%%%%%%%%%%%%%%%%%%%%%%%%%%%%%%%%%%%%%%%%%%%%%%%%%%%%%%%%%%%%%%%%%%%%%%
% PAPER INFORMATION.                                                           %
%%%%%%%%%%%%%%%%%%%%%%%%%%%%%%%%%%%%%%%%%%%%%%%%%%%%%%%%%%%%%%%%%%%%%%%%%%%%%%%%

\title{Structured barycentric forms for interpolation-based data-driven reduced
  modeling of second-order systems}
  
\author[$\ast$]{Ion Victor Gosea}
\affil[$\ast$]{
  Max Planck Institute for Dynamics of Complex Technical Systems,
  Sandtorstr. 1, 39106 Magdeburg, Germany.\authorcr
  \email{gosea@mpi-magdeburg.mpg.de}, \orcid{0000-0003-3580-4116}
}
  
\author[$\dagger$]{Serkan Gugercin}
\affil[$\dagger$]{%
  Department of Mathematics and Division of Computational Modeling and Data
  Analytics, Academy of Data Science, Virginia Tech,
  Blacksburg, VA 24061, USA.\authorcr
  \email{gugercin@vt.edu}, \orcid{0000-0003-4564-5999}
}

\author[$\ddagger$]{Steffen W. R. Werner}
\affil[$\ddagger$]{%
  Courant Institute of Mathematical Sciences, New York University,
  New York, NY 10012, USA.\authorcr
  \email{steffen.werner@nyu.edu}, \orcid{0000-0003-1667-4862}
}
  
\shorttitle{Structured barycentric forms}
\shortauthor{I.~V. Gosea, S. Gugercin, S.~W.~R. Werner}
\shortdate{2023-03-22}
\shortinstitute{}
  


\abstract{%
  An essential tool in data-driven modeling of dynamical systems from
  frequency response measurements is the barycentric form of the underlying
  rational transfer function.
  In this work, we propose structured barycentric forms for modeling 
  dynamical systems with second-order time derivatives using their frequency
  domain input-output data.
  By imposing a set of interpolation conditions, the systems' transfer
  functions are rewritten in different barycentric forms using different
  parametrizations.
  Loewner-like algorithms are developed for the explicit computation
  of second-order systems from data based on the developed barycentric forms.
  Numerical experiments show the performance of these new structured data
  driven modeling methods compared to other interpolation-based data-driven
  modeling techniques from the literature.
}

\novelty{%
  We develop new structured barycentric forms for the transfer functions of
  second-order systems that allow structured data-driven modeling from
  frequency domain input-output data.
  For the explicit computation of second-order systems from data,
  interpolation-based Loewner-like algorithms are proposed.
}

\maketitle

%%%%%%%%%%%%%%%%%%%%%%%%%%%%%%%%%%%%%%%%%%%%%%%%%%%%%%%%%%%%%%%%%%%%%%%%%%%%%%%%
% PAPER CONTENT.                                                               %
%%%%%%%%%%%%%%%%%%%%%%%%%%%%%%%%%%%%%%%%%%%%%%%%%%%%%%%%%%%%%%%%%%%%%%%%%%%%%%%%
  
\section{Introduction}%
\label{sec:intro}

Data-driven reduced-order modeling, i.e., the construction of
models describing the underlying dynamics of unknown systems from
measurements, has become an increasingly preeminent discipline.
It is an essential tool in situations when explicit models in the form of state
space formulations are not available, yet abundant input/output data are,
motivating the need for data-driven modeling.
Depending on the underlying physics, dynamical systems can
inherit differential structures leading to specific physical interpretations.
In this work, we concentrate on systems that are described by differential
equations with second-order time derivatives of the form
\begin{equation} \label{eqn:sosys}
  \begin{aligned}
    \bM \ddot{\bx}(t) + \bD \dot{\bx}(t) + \bK \bx(t) & = \bb u(t),\\
    y(t) & = \bc^{\trans} \bx(t),
  \end{aligned}
\end{equation}
with $\bM, \bD, \bK \in \R^{n \times n}$ and $\bb, \bc \in \R^{n}$.
Systems like~\cref{eqn:sosys} typically appear in the modeling of mechanical,
electrical, and related structures~\cite{AbrM87, MeyO17, Lob18, Bla18a}.
In the frequency domain (also known as the Laplace domain), the input-to-output
behavior of~\cref{eqn:sosys} is equivalently given by the corresponding
\emph{transfer function}
\begin{equation} \label{eqn:tf}
  H(s) = \bc^{\trans} (s^{2} \bM + s \bD + \bK)^{-1} \bb,
\end{equation}
which is a degree-$2n$ rational functions in $s$, where $n$ is the state-space
dimension of~\cref{eqn:sosys}. 

In recent years, several methods have been developed for learning reduced-order
state-space representations of dynamical systems from given data.
However, most of these approaches consider the classical, \emph{unstructured}
case of first-order systems of the form
\begin{equation} \label{eqn:fosys}
  \begin{aligned}
    \bE \dot{\bx}(t) & = \bA \bx(t) + \bb u(t), \\
    y(t) & = \bc^{\trans} \bx(t),
  \end{aligned}
\end{equation}
where $\bE, \bA \in \R^{n \times n}$ and $\bb, \bc \in \R^{n}$, with the
transfer function 
\begin{equation} \label{eqn:tf1st}
H(s) = \bc^{\trans} (s \bE - \bA)^{-1} \bb.
\end{equation}
Examples for such methods are the
subspace identification framework~\cite{Kun78, JuaP85, KraG18},
dynamic mode decomposition~\cite{Sch10, TuRLetal14},
operator inference~\cite{PehW16, Peh20, QiaKPetal20},
the Loewner framework~\cite{MayA07, PehGW17},
rational least-squares methods such as vector fitting~\cite{GusS99, DrmGB15} or
RKFIT~\cite{BerG17},
or the transfer-function based $\mathcal{H}_{2}$-optimal model
reduction~\cite{BeaG12}.
See also~\cite{BruK19} for a more general introduction to this topic.
The importance of preserving internal system structures in the computation of
reduced-order approximations of dynamical systems for the case of second-order
systems~\cref{eqn:sosys} has been observed in~\cite{SaaSW19, Wer21}, which
allows in particular the reinterpretation of system quantities, the preservation
of structure-inherent properties and provides cheap-to-evaluate models
with high accuracy.
However, only a few data-driven approaches have been recently extended
to~\cref{eqn:sosys} such as the Loewner framework~\cite{SchUBetal18, PonGB22}
and operator inference~\cite{FilPGetal22, ShaK22}.

In this work, we concentrate on the case in which input-output measurements are
available in the frequency domain, i.e., evaluations of the system's
transfer function~\cref{eqn:tf}.
For this type of data, the goal in data-driven reduced-order modeling is the
construction of low-order rational functions $\hH(s)$ that approximate the
given data well in an appropriate measure.
These rational functions can be interpreted as transfer functions corresponding
to dynamical systems.
Typically, it is not possible to extract additional differential structures 
from general rational functions.
For example, even though one can always convert the structured transfer
function in~\cref{eqn:tf} to an unstructured rational function
in~\cref{eqn:tf1st}, the reverse direction is not guaranteed. 
Most methods for learning transfer functions from frequency domain data 
have been mainly developed for the unstructured case~\cref{eqn:fosys}. 
In particular, such methods include the barycentric Loewner
framework~\cite{AntA86}, the vector fitting algorithm~\cite{GusS99, DrmGB15}
and the AAA algorithm~\cite{NakST18}.
The backbone of these methods is the barycentric form of rational functions,
which allows for computationally efficient constructions of rational
interpolants and least-squares fits~\cite{BerT04}.
Enforcing structures in the barycentric form allows the design of structured
data-driven modeling algorithms.
In~\cite{WerGG22}, this idea led to the extension of the vector fitting
algorithm towards mechanical systems with modal damping structure.

In this paper, we develop new structured barycentric forms associated with
the transfer functions of second-order systems~\cref{eqn:tf}.
By enforcing interpolation conditions, we show that the system matrices
in~\cref{eqn:sosys} satisfy certain equality constraints.
Using different parametrizations of the matrices in~\cref{eqn:sosys},
we derive corresponding structured barycentric forms that allow an easy
construction of the system matrices~\cref{eqn:sosys} and enforce interpolation
by construction.
We are using free parameters in the barycentric forms that are not bound in the
derivation, in order to design several Loewner-like algorithms, allowing the
direct construction of second-order systems from given
frequency domain data with interpolating transfer functions.
We also present several strategies that allow the choice of free parameters in
the structured barycentric forms to enforce additional properties in the
constructed system matrices such as positive definiteness.
Numerical examples are used to verify the developed theory and algorithms
based on these barycentric forms.

The rest of the paper is organized as follows:
In \Cref{sec:basics}, we include mathematical preliminaries, needed for the
theoretical derivations in this paper, and briefly review the theory about the
barycentric form of unstructured systems~\cref{eqn:fosys}.
Then, we develop the structured barycentric forms in 
\Cref{sec:barycentric}, followed by computational algorithms in
\Cref{sec:methods} for the explicit construction of second-order
systems~\cref{eqn:sosys} from frequency data.
\Cref{sec:examples} illustrates the effectiveness of the presented methods for
several numerical examples, including the vibrational responses of an
underwater drone and bone tissue.
The paper is concluded in \Cref{sec:conclusions}.

%%%%%%%%%%%%%%%%%%%%%%%%%%%%%%%%%%%%%%%%%%%%%%%%%%%%%%%%%%%%%%%%%%%%%%%%%%%%%%%%
% PRELIMINARIES.                                                               %
%%%%%%%%%%%%%%%%%%%%%%%%%%%%%%%%%%%%%%%%%%%%%%%%%%%%%%%%%%%%%%%%%%%%%%%%%%%%%%%%
  
\section{Mathematical preliminaries and first-order systems analysis}%
\label{sec:basics}

For our derivation of the (structured) barycentric forms, the
Sherman-Morrison-Woodbury formula for matrix inversion takes an essential role.
Given an invertible matrix $\bX \in \C^{r \times r}$ and two
vectors $\bu, \bv \in \C^{r}$ such that $\bX + \bu \bv^{\trans}$ is also
invertible, the Sherman-Morrison-Woodbury formula yields
\begin{equation} \label{eqn:smw_tmp1}
  \left( \bX + \bu \bv^{\trans} \right)^{-1} = \bX^{-1} -
    \frac{\bX^{-1}\bu \bv^{\trans}\bX^{-1}}{1 + \bv^{\trans} \bX^{-1} \bu};
\end{equation}
see, for example,~\cite{GolV13}.
In this work, we focus on transfer functions as in~\cref{eqn:tf,eqn:tf1st}
where the inverse in the middle is pre- and post-multiplied by two vectors.
Thus we consider the following adaption of~\cref{eqn:smw_tmp1}.

\begin{proposition}
  \label{prp:smw}
  Let  $\bX  \in \C^{r \times r}$ be an invertible matrix and let
  $\bu,\bv \in \C^{r}$ be  column vectors  such that
  $\bX + \bu \bv^{\trans}$ is also invertible.
  Then, for any $\bz \in \C^{r}$ it holds
  \begin{equation} \label{eqn:smw}
    \bz^{\trans} \left( \bX + \bu \bv^{\trans} \right)^{-1} \bu =
      \frac{\bz^{\trans} \bX^{-1} \bu}{1 + \bv^{\trans} \bX^{-1} \bu}.
  \end{equation}
\end{proposition}

Let $H(s)$ denote the transfer function of an unknown dynamical system.
We assume that we have access to evaluations of the transfer function at
\emph{distinct} frequency points $\lambda_{1}, \ldots, \lambda_{r} \in \C$
such that
\begin{equation} \label{eqn:data}
  \begin{aligned}
    H(\lambda_{1}) & = h_{1}, &
    H(\lambda_{2}) & = h_{2}, &
    \ldots, &&
    H(\lambda_{r}) & = h_{r}.
  \end{aligned}
\end{equation}
We denote the complete data set of frequency points and transfer
function values by $\{(\lambda_{i}, h_{i})|~ 1 \leq i \leq r\}$.

Next, we consider the parametrization of first-order (\emph{unstructured})
dynamical systems of the form~\cref{eqn:fosys} with the transfer function
$\hH(s) = \bhc^{\trans} (s \bhE - \bhA)^{-1} \bhb$ that interpolates the
given data~\cref{eqn:data}.
In the following, the first-order dynamical system~\cref{eqn:fosys} of order
$r$ is denoted by $\hSigma_{\mathrm{FO}} :(\bhE, \bhA,\bhb,\bhc)$.
A slightly different proof of the next result can be found in~\cite{AntBG20} for
the case of multi-input/multi-output dynamical systems.
For thoroughness, we include a proof here as it will be the starting
point for the structured variants considered later on. 

\begin{lemma}%
  \label{lmm:fointerp}
  Given the data~\cref{eqn:data}, define
  \begin{equation*}
    \begin{aligned}
      \bLambda & = \diag(\lambda_{1}, \ldots, \lambda_{r}) \in \C^{r \times r} &
      \text{and} && 
      \bhc & = \begin{bmatrix} h_{1} & h_{2} & \cdots & h_{r}
        \end{bmatrix}^{\trans} \in \C^{r},
    \end{aligned}
  \end{equation*}
  and let $\bones_{r}^{\trans} = \begin{bmatrix} 1 & \cdots & 1 \end{bmatrix}
  \in \C^{1 \times r}$ be the vector of ones.
  If the first-order model $\hSigma_{\mathrm{FO}}:(\bhE,\bhA,\bhb,\bhc)$
  is constructed such that
  \begin{equation} \label{eqn:paramFO}
    \bhE \bLambda - \bhA = \bhb \bones_{r}^{\trans}
  \end{equation}
  holds, where
  $\bhb = \begin{bmatrix} w_{1} & \ldots & w_{r} \end{bmatrix}^{\trans}
  \in \C^{r}$ contains free parameters with $w_{k} \neq 0$, for
  $k = 1, \ldots, r$, and the matrix $\bhE$ is invertible,
  then the transfer function 
  \begin{equation} \label{eqn:fotf}
    \hH(s) = \bhc^{\trans} (s \bhE - \bhA)^{-1} \bhb
  \end{equation}
  of $\hSigma_{\mathrm{FO}}$ interpolates the data in~\cref{eqn:data},
  i.e., it holds
  \begin{equation} \label{eqn:interpfo}
    \begin{aligned}
      \hH(\lambda_{1}) & = h_{1}, &
      \hH(\lambda_{2}) & = h_{2}, &
      \ldots, &&
      \hH(\lambda_{r}) & = h_{r}.
    \end{aligned}
  \end{equation}
\end{lemma}
\begin{proof}
  Without loss of generality, we show the proof tailored specifically to
  the case $\bhE = \bI_{r}$.
  This scenario is by no means restrictive, since the matrix $\bhE$ is
  considered to be invertible, and
  thus, can be incorporated into the matrices $\bhA$ and $\bhb$, accordingly.
  Let $\be_{i}$ denote the $i$-th unit vector of length $r$.
  By multiplying the constraint in~\cref{eqn:paramFO} with $\be_{i}$ from the
  right, one obtains
  \begin{equation*}
    (\bLambda -\bhA) \be_{i} = \bhb \bones_{r}^{\trans} \be_{i} 
  \end{equation*}
  and, therefore,
  \begin{equation} \label{eqn:fointerp_tmp}
    (\lambda_{i} \bI_{r} - \bhA) \be_{i} = \bhb.
  \end{equation}
  Note that since the entries of $\bhb$ are nonzero, $\lambda_{i}$ is not an
  eigenvalue of $\bhA$.
  We prove this claim by contradiction.
  Let the entries of $\bhb$ be nonzero and assume that $\lambda_{i}$ is an
  eigenvalue of $\bhA$ with the corresponding left eigenvector $\bv$.
  Thus, it holds that
  \begin{equation*}
    \bv^{\trans} \left( \bLambda - \bhb \bones_{r}^{\trans} - \lambda_{i}
    \bI_{r} \right) = 0,
  \end{equation*}
  and, therefore,
  \begin{equation*}
    \bv^{\trans} \left(\bLambda -  \lambda_{i} \bI_{r} \right)
    = \left(\bv^{\trans} \bhb \right) \bones_{r}^{\trans}
    = \begin{bmatrix} \bv^{\trans} \bhb & \bv^{\trans} \bhb & \ldots &
      \bv^{\trans} \bhb \end{bmatrix}.
  \end{equation*}
  Since the $i$-th entry of the row vector
  $\bv^{\trans} \left(\bLambda -  \lambda_{i} \bI_{r} \right)$ is zero,
  we consequently have  $\bv^{\trans} \bhb = 0$, and thus
  $\bv^{\trans} \left(\bLambda - \lambda_{i} \bI_{r} \right) = \boldsymbol{0}$.
  Let $\bv = \begin{bmatrix} \alpha_{1} & \alpha_{2} & \ldots & \alpha_{r}
  \end{bmatrix}^{\trans}$.
  Since $\bLambda$ is diagonal, it holds
  \begin{equation*}
    \begin{aligned}
      & \bv^{\trans} (\bLambda -  \lambda_i \bI_{r})\\
      & = \begin{bmatrix}
        \alpha_{1} (\lambda_{1} - \lambda_{i}) & \ldots & 
        \alpha_{i-1} (\lambda_{i-1} - \lambda_{i}) & 0 & 
        \alpha_{i+1} (\lambda_{i+1} - \lambda_{i}) & \ldots & 
        \alpha_{r} (\lambda_{r} - \lambda_{i}) \end{bmatrix}\\
      & = \boldsymbol{0}.
    \end{aligned}
  \end{equation*}
  Since the $\lambda_{k}$'s are assumed to be distinct, this implies 
  $\alpha_{1} = \ldots = \alpha_{i-1} = \alpha_{i+1} = \ldots = \alpha_{r} = 0$.
  Moreover, using $\bv^{\trans} \bhb = 0$, one obtains
  $\alpha_{i} w_{i} = 0$.
  But recall that $w_{i} \neq 0$; thus $\alpha_{i} = 0$ and, in summary,
  $\bv = \boldsymbol{0}$.
  However, $\bv$ is an eigenvector, which leads to the contradiction.
  Therefore, $\lambda_{i}$ is not an eigenvalue of $\bhA$.
  Since $\lambda_{i} \bI_{r} - \bhA$ is invertible, \Cref{eqn:fointerp_tmp}
  yields $(\lambda_{i} \bI_{r} - \bhA)^{-1} \bhb = \be_{i}$.
  Then by multiplying this final relation with $\bhc^{\trans}$ from the left,
  it holds $\bhc^{\trans} (\lambda_{i} \bI_{r} - \bhA)^{-1} \bhb =
  \bhc^{\trans} \be_{i}$, which proves the interpolation
  conditions~\cref{eqn:interpfo}.
\end{proof}

In this work, we do not consider the case of systems $\hSigma_{\mathrm{FO}}$
with differential-algebraic equations (descriptor systems), for which the
matrix $\bhE$ is allowed to be singular.
Such endeavors are kept for future research.
Hence, in what follows we consider the matrix $\bhE$ to be invertible.
For the simplicity of exposition, we choose without loss of
generality the matrix $\bhE$ to be the $r \times r$ identity matrix $\bI_{r}$,
since any system $\hSigma_{\mathrm{FO}}$ with $\bhE$ invertible can be
equivalently written as $(\bI_r,\bhE^{-1} \bhA,\bhE^{-1} \bhb,\bhc)$.
In this representation, one can observe that in the construction of
$\hSigma_{\mathrm{FO}}$ in~\cref{eqn:paramFO}, $r$ parameters in $\bhb$ remain
free to be chosen.
They can, for example, be used to match further $r$ interpolation conditions
additionally to~\cref{eqn:interpfo}.

Using $\bhE = \bI_{r}$, \Cref{eqn:paramFO} now reads
$\bLambda - \bhA = \bhb \bones_{r}^{\trans}$.
Thus, substituting $\bhA$ into~\cref{eqn:fotf}, the transfer function of
$\hSigma_{\mathrm{FO}}$ can be rewritten as
\begin{equation} \label{eqn:fotflr}
  \hH(s) = \bhc^{\trans} (s \bI_{r} - \bhA)^{-1} \bhb
  = \bhc^{\trans} \left[ s \bI_{r} - (\bLambda - \bhb \bones_{r}^{\trans})
    \right]^{-1} \bhb
  = \bhc^{\trans} \left[(s \bI_{r} - \bLambda) + \bhb \bones_{r}^{\trans}
    \right]^{-1} \bhb.
\end{equation}
Define $\bhPhi(s) = s \bI_{r} - \bLambda$ to be the diagonal matrix function
depending on the frequency parameter $s \in \C$.
Then, the transfer function in~\cref{eqn:fotflr} can be formulated as
\begin{equation} \label{eqn:fophi}
  \hH(s) = \bhc^{\trans} (s \bI_{r} - \bhA)^{-1} \bhb
  = \bhc^{\trans} \left(\bhPhi(s) + \bhb \bones_{r}^{\trans} \right)^{-1} \bhb.
\end{equation}
The form of the transfer function $\hH(s)$ in terms of $\bhPhi(s)$ as given
by~\cref{eqn:fophi} will play a crucial role in later sections to extend the
interpolation theory to the structured case.
The following result, which recovers the classical barycentric form of rational
interpolants, follows from applying \Cref{lmm:fointerp} to~\cref{eqn:fophi}.

\begin{corollary}
  \label{cor:fobary}
  Given the interpolation data~\cref{eqn:data}, the transfer
  function~\cref{eqn:fophi} of the first-order model that yields
  the interpolation conditions~\cref{eqn:interpfo} can be equivalently
  expressed as
  \begin{equation} \label{eqn:fotfbary1}
     \hH(s) = \frac{\bhc^{\trans} \bhPhi(s)^{-1} \bhb}{1 + \bones_{r}^{\trans}
       \bhPhi(s)^{-1} \bhb},
  \end{equation}
  where $\bhPhi(s) = s \bI_{r} - \bLambda$, and $\bLambda$, $\bhc$, and $\bhb$
  are defined as in \Cref{lmm:fointerp}.
  This formula can be further represented as barycentric rational
  interpolation form
   \begin{equation} \label{eqn:fotfbary2}
     \hH(s) = \frac{\displaystyle \sum_{i = 1}^{r}
       \frac{h_{i} w_{i}}{s - \lambda_{i}}}
       {\displaystyle 1 + \sum_{i = 1}^{r} \frac{w_{i}}{s-\lambda_{i}}}.
  \end{equation}
\end{corollary}
\begin{proof}
  Applying the identity~\cref{eqn:smw} in \Cref{prp:smw} to
  $\hH(s) = \bhc^{\trans} \left(\bhPhi(s) + \bhb \bones_{r}^{\trans}
  \right)^{-1} \bhb$ using
  \begin{equation*}
    \begin{aligned}
      \bu & = \bhb, &
      \bv & = \bones_{r}, &
      \bz & = \bhc, &
      \bX & = \bhPhi(s)
    \end{aligned}
  \end{equation*}
  yields~\cref{eqn:fotfbary1}.
  The fact that $\bhPhi(s)$ is diagonal and the definitions of 
  $\bhb$ and $\bhc$ directly lead to
  \begin{equation*}
    \begin{aligned}
      \bhc^{\trans} \bhPhi(s)^{-1} \bhb & = \sum\limits_{i = 1}^{r}
        \frac{h_{i} w_{i}}{(s - \lambda_{i})} & \text{and} &&
      \bones_{r}^{\trans} \bhPhi(s)^{-1} \bhb = \sum\limits_{i = 1}^{r}
        \frac{w_{i}}{(s - \lambda_{i})},
    \end{aligned}
  \end{equation*}
  which then together result in the barycentric form~\cref{eqn:fotfbary2}.
\end{proof}

As stated earlier, the expression~\cref{eqn:fotfbary2} is known as the
barycentric form of the rational interpolant and is a well-studied
object~\cite{BerT04} as it forms the foundation for many rational
approximation techniques~\cite{AntA86, NakST18}.
The derivation of the barycentric form in \Cref{cor:fobary} follows a
perspective from systems and control theory that aligns well with
the second-order dynamics we study next.

The additional value of one in the denominator of~\cref{eqn:fotfbary1} appears
as a result of the Sherman-Morrison-Woodbury formula.
Typically, in the classical theory about barycentric interpolation, 
such term does not appear in the denominator of the barycentric formula; see in
particular~\cite{BerT04, NakST18}.
The rational function represented in~\cref{eqn:fotfbary1} is strictly proper
since the degree of the denominator is greater than the one of the numerator,
which aligns with the setting of corresponding LTI systems.
Although this may seem restrictive, the proposed approach can also accommodate
proper rational functions corresponding to LTI systems with a nonzero
feed-through term in the state-output equation.

%%%%%%%%%%%%%%%%%%%%%%%%%%%%%%%%%%%%%%%%%%%%%%%%%%%%%%%%%%%%%%%%%%%%%%%%%%%%%%%%
% STRUCTURED BARYCENTRIC FORMS.                                                %
%%%%%%%%%%%%%%%%%%%%%%%%%%%%%%%%%%%%%%%%%%%%%%%%%%%%%%%%%%%%%%%%%%%%%%%%%%%%%%%%
  
\section{Structured barycentric forms}%
\label{sec:barycentric}

As in the previous section, we assume to have transfer function measurements
of the form~\cref{eqn:data} given and aim to construct models that fit the
given data.
However, in contrast to \Cref{sec:basics}, we aim, from now on, to 
construct structured models of the form~\cref{eqn:sosys} denoted as
$\hSigma_{\mathrm{SO}}: (\bhM, \bhD, \bhK, \bhb, \bhc)$, with the model matrices
$\bhM, \bhD, \bhK \in \R^{r \times r}$ and $\bhb, \bhc \in \R^{r}$.
Before we present the main results of this work, we introduce the following two
sets of assumptions that will be needed later on.

\begin{assumption} \label{asm:one}
  For the model matrices $\hSigma_{\mathrm{SO}}:
  (\bhM, \bhD, \bhK, \bhb, \bhc)$ and given data in~\cref{eqn:data}, we assume
  that
  \begin{itemize}
    \item[(a)] the matrix $\bhM$ is invertible, and
      \hfill (A1.1)
    \item[(b)] the interpolation points $\{\lambda_{1}, \ldots, \lambda_{r} \}$
      are all distinct.
      \hfill (A1.2)
  \end{itemize}
\end{assumption}

The reasons for imposing \hyperref[asm:one]{Assumptions~(A1.1)}
and~\hyperref[asm:one]{(A1.2)} are similar to those in the case of
first-order systems from the previous section.
More specifically, \hyperref[asm:one]{Assumption~(A1.1)} enforces the
system~\cref{eqn:sosys} to be described by ordinary differential equations
rather than differential-algebraic ones, which require a singular $\bhM$.
The modeling of such descriptor systems is left for future research.
As in the first-order case, \hyperref[asm:one]{Assumption~(A1.2)} is
necessary to avoid inconsistencies in the interpolation conditions.
The case in which repeated interpolation points and derivative data are
used for Hermite interpolation will be considered in a separate work.

\begin{assumption} \label{asm:two}
  For the model matrices $\hSigma_{\mathrm{SO}}:
  (\bhM, \bhD, \bhK, \bhb, \bhc)$ and given data in~\cref{eqn:data}, we assume
  for $i, k = 1, \ldots, r$ and $i \neq k$ that either
  \begin{itemize}
    \item[(a)] $-(\lambda_{k} + \lambda_{i})$ is not an eigenvalue of
      the matrix $\bhM^{-1}\bhD$, or
      \hfill (A2.1)
    \item[(b)] $(\lambda_{k} \lambda_{i})$ is not an eigenvalue of
      the matrix $\bhM^{-1}\bhK$.
      \hfill (A2.2)
  \end{itemize}
\end{assumption}

In contrast to \hyperref[asm:one]{Assumptions~(A1.1)}
and~\hyperref[asm:one]{(A1.2)}, which need to hold both at the same time,
only one of \hyperref[asm:two]{Assumptions~(A2.1)}
and~\hyperref[asm:two]{(A2.2)} will be imposed at once, since these
two assumptions are equivalent to each other for their respectively
corresponding structured barycentric form, as it will become clearer later on. 
Although  \hyperref[asm:two]{Assumptions~(A2.1)}
and~\hyperref[asm:two]{(A2.2)} may seem restrictive at first glance, we will
show that they occur naturally for practical choices of the parameters in
the new structured barycentric forms.
A more detailed discussion of this topic is provided in \Cref{subsec:linparam}.

%%%%%%%%%%%%%%%%%%%%%%%%%%%%%%%%%%%%%%%%%%%%%%%%%%%%%%%%%%%%%%%%%%%%%%%%%%%%%%%%

\subsection{Interpolatory second-order transfer functions}%
\label{subsec:interpso}

The following result extends \Cref{lmm:fointerp} to second-order
systems establishing sufficient conditions for the interpolation of given
transfer function data~\cref{eqn:data}.

\begin{lemma}%
  \label{lmm:sointerp}
  Given the interpolation data~\cref{eqn:data}, define
  \begin{equation*}
    \begin{aligned}
      \bLambda & = \diag(\lambda_{1}, \ldots, \lambda_{r}) \in \C^{r \times r} &
      \text{and} && 
      \bhc & = \begin{bmatrix} h_{1} & h_{2} & \cdots & h_{r}
        \end{bmatrix}^{\trans} \in \C^{r},
    \end{aligned}
  \end{equation*}
  and let $\bones_{r}^{\trans} = \begin{bmatrix} 1 & \cdots & 1 \end{bmatrix}
  \in \C^{1 \times r}$ be the vector of ones.
  Let the second-order model $\hSigma_{\mathrm{SO}}:
  (\bhM,\bhD, \bhK, \bhb, \bhc)$ be constructed such that
  \begin{equation} \label{eqn:paramSO}
    \bhM \bLambda^{2} + \bhD \bLambda + \bhK = \bhb \bones_{r}^{\trans},
  \end{equation}
  holds, where $\bhb = \begin{bmatrix} w_{1} & \ldots & w_{r}
  \end{bmatrix}^{\trans} \in \C^{r}$ contains free parameters with
  $w_{k} \neq 0$ for $k = 1, \ldots, r$.
  If \hyperref[asm:one]{Assumptions~(A1.1)} and~\hyperref[asm:one]{(A1.2)} as
  well as either \hyperref[asm:two]{Assumption~(A2.1)}
  or~\hyperref[asm:two]{(A2.2)} hold, then the transfer function 
  \begin{equation*}
    \hH(s) = \bhc^{\trans} (s^{2} \bhM + s \bhD + \bhK)^{-1} \bhb
  \end{equation*}
  of $\hSigma_{\mathrm{SO}}$ interpolates the data in~\cref{eqn:data}, i.e.,
  it holds
  \begin{equation} \label{eqn:interpso}
    \begin{aligned}
      \hH(\lambda_{1}) & = h_{1}, &
      \hH(\lambda_{2}) & = h_{2}, &
      \ldots, &&
      \hH(\lambda_{r}) & = h_{r}.
    \end{aligned}
  \end{equation}
\end{lemma}
\begin{proof}
  Without loss of generality, we show the proof tailored specifically to
  the case $\bhM = \bI_{r}$.
  This scenario is by no means restrictive, since the $\bhM$ matrix is
  considered to be invertible; see \hyperref[asm:one]{Assumption~(A1.1)}.
  Thus, it can be incorporated into $\bhD$, $\bhK$ and $\bhb$, accordingly.
  For the consideration of eigenvalues of second-order systems, we
  introduce the augmented matrices
  \begin{equation} \label{eqn:sointerpTmp1}
    \begin{aligned}
      \bcE & = \begin{bmatrix}
        \bI_{r} & 0 \\ 0 & \bhM
      \end{bmatrix} = \bI_{2r}, &
      \bcA & = \begin{bmatrix}
        0 & \bI_{r} \\ -\bhK & -\bhD
      \end{bmatrix}.
    \end{aligned}
  \end{equation}
  Since the entries of $\bhb$ are nonzero and with
  \hyperref[asm:two]{Assumption~(A2.1)} that $-(\lambda_{k} + \lambda_{i})$ is
  not an eigenvalue of matrix $\bhD$, the interpolation point $\lambda_{i}$ is
  not a solution to the linearized eigenvalue problem of the matrix pencil
  $(\bcA,\bcE)$ in~\cref{eqn:sointerpTmp1}, i.e., it is not an eigenvalue of the
  quadratic pencil involving  $\bhM$, $\bhD$ and $\bhK$.
  We prove this claim by contradiction.
  Let the entries of $\bhb$ be nonzero and assume that $\lambda_{i}$ is an
  eigenvalue of $\bcA$ with the corresponding left-eigenvector
  $\bv^{\trans} = \begin{bmatrix} \bv_{1}^{\trans} & \bv_{2}^{\trans}
  \end{bmatrix}$.
  Thus, it holds
  \begin{equation*}
    \begin{bmatrix} \bv_{1}^{\trans} & \bv_{2}^{\trans} \end{bmatrix}
      \left( \lambda_{i} \bcE  - \bcA  \right) = \bfz.
  \end{equation*}
  Employing the block matrix structure from~\cref{eqn:sointerpTmp1}
  yields the quadratic eigenvalue relation
  \begin{equation} \label{eqn:sointerpTmp2}
    \lambda_{i}^{2} \bv_{2}^{\trans} \bI_{r} +
      \lambda_{i} \bv_{2}^{\trans} \bhD +
      \bv_{2}^{\trans} \bhK = \bfz.
  \end{equation}
  By multiplying the constraint in~\cref{eqn:paramSO} with $\be_{i}$ from the
  right, we obtain for $1 \leq i \leq k$ that
  \begin{equation} \label{eqn:sointerpTmp3}
    (\lambda_{i}^{2} \bI_{r} + \bhD \lambda_{i} + \bhK) \be_{i} = \bhb.
  \end{equation}
 Then,  multiplication of this last equation~\cref{eqn:sointerpTmp3} with
  $\bv_{2}^{\trans}$ from the left yields
  \begin{equation} \label{eqn:sointerpTmp4}
    \begin{aligned}
      \lambda_{i}^{2} \bv_{2}^{\trans} \be_{i} +
        \lambda_{i} \bv_{2}^{\trans} \bhD \be_{i} +
        \bv_{2}^{\trans} \bhK \be_{i} = \bv_{2}^{\trans} \bhb.
    \end{aligned}
  \end{equation}
  It follows directly from~\cref{eqn:sointerpTmp2,eqn:sointerpTmp4}
  that $\bv_{2}^{\trans} \bhb = \bfz$.
  Let the eigenvector be given as
  $\bv_{2} = \begin{bmatrix} \alpha_{1} & \alpha_{2} & \ldots & \alpha_{r}
  \end{bmatrix}^{\trans}$.
  Without loss of generality assume that $\bhD$ is a diagonal matrix with
  $\bhD = \diag(\delta_{1}, \ldots, \delta_{r})$.
  Now we use~\cref{eqn:paramSO} to describe the stiffness matrix $\bhK$ in
  terms of the rest such that $\bhK = \bhb \bones_{r}^{\trans} - \bhD
  \bLambda- \bLambda^{2}$ and substitute this relation
  into~\cref{eqn:sointerpTmp2} to obtain
  \begin{align*}
    \bfz & = \lambda_{i}^{2} \bv_{2}^{\trans} +
      \lambda_{i} \bv_{2}^{\trans} \bhD +
      \bv_{2}^{\trans} (\bhb \bones_{r}^{\trans} - \bhD \bLambda -
      \bLambda^{2}) \\
    & = \bv_{2}^{\trans} \bhD (\lambda_{i} \bI_{r}- \bLambda) +
      \bv_{2}^{\trans} (\lambda_{i}^{2} \bI_{r} - \bLambda^{2}) +
      \underbrace{\bv_{2}^{\trans} \bhb}_{\phantom{\,\bfz}=\,\bfz}
      \bones_{r}^{\trans} \\
    & = \bv_{2}^{\trans}(\bhD + \lambda_{i} \bI_{r} + \bLambda)
      (\lambda_{i} \bI_{r} - \bLambda) \\
    & = \begin{bmatrix} \alpha_{1} (\delta_{1} + \lambda_{1} + \lambda_{i})
      (\lambda_{1} - \lambda_{i}) & \ldots  & 0  & \ldots & 
      \alpha_{r} (\delta_{r} + \lambda_{r} + \lambda_{i})
      (\lambda_{r} - \lambda_{i}) \end{bmatrix}.
  \end{align*}
  Since the $\lambda_{k}$'s are distinct (\hyperref[asm:one]{Assumption~(A1.2)})
  and $\delta_{k} + \lambda_{k} + \lambda_{i} \neq 0$ for all $1 \leq k \leq r$
  due to \hyperref[asm:two]{Assumption~(A2.1)}, it implies that 
  $\alpha_{1} = \ldots = \alpha_{i-1} = \alpha_{i+1} = \ldots = \alpha_{r} = 0$.
  Moreover, using $\bv_{2}^{\trans} \bhb = 0$, it holds that
  $\alpha_{i} w_{i} = 0$.
  Since $w_{i} \neq 0$, this would imply $\alpha_{i} = 0$, yielding
  $\bv_{2} = \bfz$.
  However, $\bv_2$ is an eigenvector, thus leading to the contradiction.
  Therefore, $\lambda_{i}$ is not a solution to the quadratic eigenvalue
  problem.

  As a results, the matrix $\lambda_{i}^{2} \bI_{r} + \lambda_{i} \bhD + \bhK$
  is non-singular, and by multiplying~\cref{eqn:sointerpTmp3} with
  $(\bI_{r} \lambda_{i}^{2} + \bhD \lambda_{i} + \bhK)^{-1}$ from the left,
  we get
  \begin{equation*}
    (\bI_r \lambda_{i}^{2} + \bhD \lambda_{i} + \bhK)^{-1} \bhb = \be_{i}
  \end{equation*}
  and, therefore,
  \begin{equation*}
    \bhc^{\trans} (\bI_r \lambda_{i}^{2} + \bhD \lambda_{i} + \bhK)^{-1} \bhb =
      \bhc \be_{i} = h_{i}.
  \end{equation*}
  Hence, we have shown that $\hH(\lambda_{i}) = h_{i}$ for any
  $1 \leq i \leq r$.
  Note that we only used \hyperref[asm:two]{(A2.1)} out of \Cref{asm:two}
  in this proof, which allowed the description of the stiffness matrix $\bhK$
  by the other terms in~\cref{eqn:paramSO}.
  An analogous proof relies on \hyperref[asm:two]{Assumption~(A2.2)}, which
  allows the reformulation of~\cref{eqn:paramSO} for the damping matrix
  $\bhD$.
  Due to the similarity to the presented proof, we omit this part.
\end{proof}

Similar to the case of first-order systems in \Cref{sec:basics}, we can
eliminate one of the unknown matrices in the constraint~\cref{eqn:paramSO}
by using \hyperref[asm:one]{Assumption~(A1.1)}.
Thereby, we will choose the mass matrix to be the $r$-dimensional
identity matrix, $\bhM = \bI_{r}$.
However, in contrast to the case of first-order systems, this leaves us with
three remaining unknown matrices in~\cref{eqn:paramSO} instead of two.
Following diagonalization assumptions, which we will point out later in detail,
this leaves us with $2r$ free parameters to choose for the explicit realization
of interpolating second-order systems.

\begin{remark}
  Aside from this work, a data-driven method for the derivation of structured
  models with interpolating transfer functions has been developed
  in~\cite{SchUBetal18}.
  Therein, the authors use constraints similar to~\cref{eqn:paramSO} to
  parametrize the model matrices as the solution of large-scale linear systems 
  of equations to enforce the interpolation conditions.
  The unknowns in these linear systems correspond to the entries of the
  vectorized state-space quantities.
  There is no discussion of or connection to structured barycentric forms
  in~\cite{SchUBetal18} (which represents the main novelty of the current work)
  as~\cite{SchUBetal18} is directly related to and based on the
  \emph{non-barycentric} Loewner framework for
  interpolation~\cite{MayA07} and the projection-based interpolatory model
  reduction of structured systems~\cite{BeaG09}.
  However, it will be interesting to revisit~\cite{SchUBetal18} in a future
  work since it might provide directions for extending the barycentric form to
  different structures than those we consider here.  
\end{remark}

In the following derivation of structured barycentric forms, we will make use
of the equality constraint in~\cref{eqn:paramSO} that enforces $r$ interpolation
conditions.
As mentioned above, the remaining free $2r$ parameters are given in the input
vector $\bhb$ and either in the stiffness matrix $\bhK$ or damping
matrix $\bhD$.
Therefore, different barycentric forms result from the reformulation
of~\cref{eqn:paramSO} in terms of either stiffness (in \Cref{subsec:constrK})
or damping matrix (in \Cref{subsec:constrD}).

%%%%%%%%%%%%%%%%%%%%%%%%%%%%%%%%%%%%%%%%%%%%%%%%%%%%%%%%%%%%%%%%%%%%%%%%%%%%%%%%

\subsection{Parametrization with constrained stiffness matrix}
\label{subsec:constrK}

%%%%%%%%%%%%%%%%%%%%%%%%%%%%%%%%%%%%%%%%

\subsubsection{General setup}

In this section, we incorporate the remaining free parameters of the system
$\hSigma_{\mathrm{SO}}$ into the damping matrix $\bhD$ and the input
vector $\bhb$.
Therefore, from~\cref{eqn:paramSO} it follows that the stiffness matrix
satisfies
\begin{equation} \label{eqn:sylvK}
  \bhK = \bhb \bones_{r}^{\trans}-\bhM \bLambda^{2} - \bhD \bLambda.
\end{equation}
Substituting~\cref{eqn:sylvK} into the transfer function $\hH(s)$ corresponding
to the second-order model $\hSigma_{\mathrm{SO}}: (\bhM,\bhD, \bhK, \bhb,
\bhc)$ yields
\begin{align} \label{eqn:sophi}
  \hH(s) & = \bhc^{\trans} (s^2\bhM + s \bhD+\bhK)^{-1} \bhb \nonumber \\
  & = \bhc^{\trans} (s^{2} \bhM + s \bhD + \bhb \bones_{r}^{\trans} -
    \bhM \bLambda^{2} - \bhD \bLambda)^{-1} \bhb \nonumber\\
  & = \bhc^{\trans} (\bhPhi(s) + \bhb \bones_{r}^{\trans})^{-1} \bhb,
\end{align}
where the matrix-valued function $\bhPhi(s)$ is given by
\begin{equation*}
  \begin{aligned}
    \bhPhi(s) & = s^{2} \bhM + s \bhD - \bhM \bLambda^{2} - \bhD \bLambda\\
    & = \bhM (s^{2} \bI_{r}-\bLambda^{2})+\bhD (s \bI_{r} -\bLambda) \\
    & =  \left( \bhM(s \bI_{r}+\bLambda) + \bhD  \right) (s \bI_{r} - \bLambda).
  \end{aligned}
\end{equation*}

The following lemma states the structured barycentric form of~\cref{eqn:sophi}
in terms of the input and output vectors, and the matrix-valued function
$\bhPhi(s)$.

\begin{lemma}%
  \label{lmm:sophi}
  Given the interpolation data~\cref{eqn:data}, the transfer
  function~\cref{eqn:sophi} of the second-order model that yields the
  interpolation conditions~\cref{eqn:interpso} can be equivalently expressed as
  \begin{equation*}
     \hH(s) = \frac{\bhc^{\trans} \bhPhi(s)^{-1} \bhb}{1 + \bones_{r}^{\trans}
       \bhPhi(s)^{-1} \bhb},
  \end{equation*}
  where $\bhPhi(s) = \left( \bhM(s \bI_{r}+\bLambda) + \bhD  \right)
  (s \bI_{r} - \bLambda)$, and $\bLambda$, $\bhc$, and $\bhb$ are as defined in
  \Cref{lmm:fointerp}. 
\end{lemma}
\begin{proof}
  Using \Cref{prp:smw} for the formulation of $\hH(s)$ in~\cref{eqn:sophi}, with
  the following choice of vectors and matrices
  \begin{equation*}
    \begin{aligned}
      \bu & = \bhb, &
      \bv & = \bones_{r}, &
      \bz & = \bhc, & 
      \bX & = \bhPhi(s)
    \end{aligned}
  \end{equation*}
  yields the result of the lemma.
\end{proof}

As mentioned in \Cref{subsec:interpso}, we choose the mass matrix $\bhM$ to be
the identity due to \hyperref[asm:one]{Assumption~(A1.1)}.
Under the assumption that $\bhD$ has no higher order Jordan blocks, it can be
diagonalized while preserving $\bhM = \bI_{r}$ such that
\begin{equation*}
  \begin{aligned}
    \bhD & = \diag(\delta_{1}, \ldots, \delta_{r}) & \text{and} &&
    \bhM & = \bI_{r} = \diag(1, \ldots, 1).
  \end{aligned}
\end{equation*}
In this case, the matrix-valued function $\bhPhi(s)$ is a diagonal matrix for
all $s \in \C$, which allows us to write
\begin{align} \label{eqn:hphiK}
  \bhPhi(s) & = \left( \bhM(s \bI_{r} + \bLambda) + \bhD \right)
    (s \bI_{r} - \bLambda) \nonumber \\
  & = \diag \Big( (s - \lambda_{1}) (s + \lambda_{1} + \delta_{1}), \ldots,
    (s - \lambda_{r}) (s + \lambda_{r} + \delta_{r}) \Big).
\end{align}
Using this diagonal form of the matrix-valued function in \Cref{lmm:sophi}
yields the following result:
a structured barycentric formula for second-order transfer functions.
\begin{theorem}%
  \label{thm:baryformK}
  Given interpolation points and measurements
  $\{(\lambda_{i}, h_{i})|~ 1 \leq i \leq r\}$ and let
  \hyperref[asm:one]{Assumptions~(A1.1)}
  and~\hyperref[asm:one]{(A1.2)} as well as
  \hyperref[asm:two]{Assumption~(A2.1)} hold.
  The barycentric form of the transfer function $\hH(s)$ corresponding to
  second-order system $\hSigma_{\mathrm{SO}}: (\bhM,\bhD, \bhK, \bhb, \bhc)$ is
  given by
  \begin{equation} \label{eqn:barytfK}
    \hH(s) = \frac{\displaystyle \sum_{i = 1}^{r}
      \frac{h_{i} w_{i}}{(s - \lambda_{i})(s - \sigma_{i})}}
      {\displaystyle 1 + \sum_{i = 1}^{r}
      \frac{w_{i}}{(s - \lambda_{i})(s - \sigma_{i})}},
  \end{equation}
  with the weights $0 \neq w_{i} \in \C$ and support points
  $\sigma_{i} = -(\delta_{i} + \lambda_{i})$, where $\delta_{i} \in \C$ are
  damping parameters, for $1 \leq i \leq r$.
  The barycentric form~\cref{eqn:barytfK} satisfies the interpolation
  conditions~\cref{eqn:interpso}.
  The matrices of the corresponding second-order system are given by
  \begin{equation} \label{eqn:sosyswithK}
    \begin{aligned}
      \bhM & = \bI_{r}, &
      \bhD & = -\diag(\lambda_1+\sigma_1, \ldots, \lambda_r+\sigma_r),\\
      \bhK & = \bhb \bones_{r}^{\trans} - \bLambda^{2} - \bhD \bLambda, &
      \bhb & = \begin{bmatrix} w_{1} & \ldots & w_{r}
        \end{bmatrix}^{\trans}, \\
      \bhc & = \begin{bmatrix} h_{1}  & \ldots & h_{r} \end{bmatrix}^{\trans}.
    \end{aligned}
  \end{equation}
\end{theorem}
\begin{proof}
  By making use of the diagonal structure of $\bhPhi(s)$ in~\cref{eqn:hphiK}
  and the other components 
  \begin{equation*}
    \begin{aligned}
      \bones_{r} & = \begin{bmatrix} 1 & \ldots & 1  \end{bmatrix}^{\trans}, &
      \bhb & = \begin{bmatrix} w_{1} & \ldots & w_{r} \end{bmatrix}^{\trans}, &
      \bhc & = \begin{bmatrix} h_{1} & \ldots & h_{r} \end{bmatrix}^{\trans}
    \end{aligned}
  \end{equation*}
  in the formulation of the transfer function in \Cref{lmm:sophi},
  the transfer function is rewritten in barycentric form by multiplying out
  the matrix-vector products as
  \begin{equation*}
    \hH(s) = \frac{\bhc^{\trans} \bhPhi(s)^{-1} \bhb}{1 + \bones_{r}^{\trans}
      \bhPhi(s)^{-1} \bhb}
    = \frac{\displaystyle \sum_{i = 1}^{r} \frac{h_{i} w_{i}}
      {(s - \lambda_{i})(s +  \lambda_{i} + \delta_{i})}}
      {\displaystyle 1 + \sum_{i = 1}^{r} \frac{w_{i}}{(s - \lambda_{i})
      (s + \lambda_{i} + \delta_{i})}}.
  \end{equation*}
  Setting the support points $\sigma_{i}  = -(\delta_{i} +  \lambda_{i})$, the
  result in~\cref{eqn:barytfK} follows directly.
  The realization~\cref{eqn:sosyswithK} is then given by rearranging the
  different parameters into the corresponding matrices and vectors.
\end{proof}

Note that given the notation $\bSigma = \diag(\sigma_{1}, \ldots, \sigma_{r})$,
the realization in~\cref{eqn:sosyswithK} can equivalently be written as
\begin{equation*}
  \begin{aligned}
    \bhM & = \bI_{r}, &
    \bhD & = -\bLambda-\bSigma, \\
    \bhK & = \bhb \bones_{r}^{\trans} + \bLambda \bSigma, &
    \bhb & = \begin{bmatrix} w_{1}  & \ldots & w_{r}
        \end{bmatrix}^{\trans}, \\
    \bhc & = \begin{bmatrix} h_{1}  & \ldots & h_{r} \end{bmatrix}^{\trans}.
  \end{aligned}
\end{equation*}
The free parameters that explicitly appear above are $2r$ in total and are
given by the entries of the vector $\bhb$ and of the diagonal matrix
$\bSigma$, i.e., the free parameters in the structured barycentric
form~\cref{eqn:barytfK} are $\{w_{1}, \ldots, w_{r} \}  \cup 
\{\sigma_{1}, \cdots, \sigma_{r} \}$.

%%%%%%%%%%%%%%%%%%%%%%%%%%%%%%%%%%%%%%%%

\subsubsection{Systems with zero damping matrix}%
\label{subsubsec:constrKD0}

An important subclass of second-order systems~\cref{eqn:sosys} is given by a
zero damping matrix, i.e., $\bD = 0$.
These occur, for example, in the case of ``conservative'' dynamics where no
dissipation/damping is considered.
Hamiltonian systems belong to this category~\cite{AbrM87, MeyO17}.
Retaining this additional structure allows, for example, modeling the
preservation of energy in the system.
Another problem class that can be modeled by a zero damping matrix is the
case of hysteretic damping, i.e., constant damping over the complete frequency
range~\cite{CheY97, AumW23}.
This is used, for example, to model the general influence of physical structures
on the damping behavior of systems.
Thereby, the damping matrix is considered to be frequency dependent with
$\bD(s) = \frac{1}{s} \i \eta K$.
Inserting this damping definition into the second-order transfer
function~\cref{eqn:tf} yields
\begin{equation*}
  H(s) = \bc^{\trans} \left( s^{2} \bM + \frac{s}{s} \i \eta \bK + \bK
    \right)^{-1} \bb
    = \bc^{\trans} (s^{2} \bM + (1 + \i \eta) \bK)^{-1} \bb,
\end{equation*}
which can be seen as a system with a complex stiffness matrix and zero damping
matrix.
The following corollary refines the results from \Cref{thm:baryformK} to the
case of system structure with zero damping matrix, $\bhD = \bfz$.

\begin{corollary}%
  \label{cor:barytfKD0}
  Given interpolation points and measurements
  $\{(\lambda_{i}, h_{i})|~1 \leq i \leq r\}$ and let
  \hyperref[asm:one]{Assumptions~(A1.1)}
  and~\hyperref[asm:one]{(A1.2)} as well as
  \hyperref[asm:two]{Assumption~(A2.1)} hold.
  The barycentric form of the transfer function $\hH(s)$ corresponding to
  second-order system $\hSigma_{\mathrm{SO}}: (\bhM, \bfz, \bhK, \bhb, \bhc)$ is
  given by
  \begin{equation} \label{eqn:barytfKD0}
     \hH(s) = \frac{\displaystyle \sum_{i = 1}^{r}
       \frac{h_{i} w_{i}}{s^{2} - \lambda_{i}^{2}}}
       {\displaystyle 1 + \sum_{i = 1}^{r}
       \frac{w_{i}}{s^{2} - \lambda_{i}^{2}}},
  \end{equation}
  with the weights $0 \neq w_{i} \in \C$, for $1 \leq i \leq r$.
  The barycentric form~\cref{eqn:barytfKD0} satisfies the interpolation
  conditions~\cref{eqn:interpso} and it can be written as a second-order
  dynamical systems with zero damping, i.e.,
  \begin{equation*}
    \hH(s) = \bhc^{\trans} (s^2\bhM + \bhK)^{-1} \bhb,
  \end{equation*}
  where 
  \begin{equation*}
    \begin{aligned}
      \bhM & = \bI_{r}, &
      \bhK & = \bhb \bones_{r}^{\trans} - \bLambda^{2}, &
      \bhb &= \begin{bmatrix} w_{1} & \ldots & w_{r} \end{bmatrix}^{\trans}, &
      \bhc & = \begin{bmatrix} h_{1} & \ldots & h_{r} \end{bmatrix}^{\trans}.
    \end{aligned}
  \end{equation*}
\end{corollary}
As in \Cref{thm:baryformK}, \hyperref[asm:two]{Assumption~(A2.1)} needs to
hold for~\cref{eqn:barytfKD0} to satisfy the interpolation
conditions~\cref{eqn:interpso}.
However, in the special case of $\bhD = \bfz$,
\hyperref[asm:two]{Assumption~(A2.1)} simplifies to
$\lambda_{i} \neq -\lambda_{k}$, for all $i \neq k$.
In particular, if the interpolation points are chosen on the imaginary
axis, no complex conjugate pairs are allowed in the set of
interpolation points.

%%%%%%%%%%%%%%%%%%%%%%%%%%%%%%%%%%%%%%%%%%%%%%%%%%%%%%%%%%%%%%%%%%%%%%%%%%%%%%%%

\subsection{Parametrization with constrained damping matrix}%
\label{subsec:constrD}

In this section, we incorporate the remaining free parameters
of the system $\hSigma_{\mathrm{SO}}$ into the stiffness matrix $\bhK$ and the
input vector $\bhb$.
Therefore, it follows from~\cref{eqn:paramSO} that  the damping matrix satisfies
\begin{equation} \label{eqn:sylvDtmp}
  \bhD \bLambda = \bhb \bones_{r}^{\trans} - \bhM \bLambda^{2} - \bhK.
\end{equation}
Under the assumption that $\bLambda$ is invertible, i.e., zero is not an
interpolation point, one can equivalently write~\cref{eqn:sylvDtmp} as
\begin{equation} \label{eqn:sylvD}
  \bhD = \bhb \bones_{r}^{\trans} \bLambda^{-1} - \bhM \bLambda -
    \bhK \bLambda^{-1}.
\end{equation}
By substituting~\cref{eqn:sylvD} into the formula of the transfer function
$\hH(s)$ corresponding to the second-order model
$\hSigma_{\mathrm{SO}}: (\bhM,\bhD, \bhK, \bhb, \bhc)$, it holds
\begin{align} \label{eqn:sopsi}
  \hH(s) & = \bhc (s^{2}\bhM + s \bhD + \bhK)^{-1} \bhb  \nonumber \\
  & = \bhc^{\trans} \left( s^{2} \bhM + s \left( \bhb \bones_{r}^{\trans}
    \bLambda^{-1} - \bhM \bLambda - \bhK \bLambda^{-1} \right) +
    \bhK \right)^{-1} \bhb \nonumber \\
  & = \bhc^{\trans} \left( \bhM (s^{2} \bI_{r} - s \bLambda) + \bhK (\bI_{r} -
    s \bLambda^{-1}) + s \bhb \bones_{r}^{\trans} \bLambda^{-1}
    \right)^{-1} \bhb \nonumber \\
  & = \bhc^{\trans} \left( \bhPsi(s) + s \bhb \bones_{r}^{\trans}
    \bLambda^{-1} \right)^{-1} \bhb,
\end{align}
where the matrix-valued function $\bhPsi(s)$ is given by
\begin{equation*}
  \begin{aligned}
    \bhPsi(s) & = \bhM (s^{2} \bI_{r} - s \bLambda) + \bhK
      (\bI_{r} - s \bLambda^{-1}) \\
    & = (s \bhM  - \bhK \bLambda^{-1}) (s \bI_{r} -\bLambda).
  \end{aligned}
\end{equation*}

Similar to \Cref{lmm:sophi}, the following lemma states the structured
barycentric form of~\cref{eqn:sopsi} in terms of the input and output vectors
and the matrix-valued function $\bhPhi(s)$.

\begin{lemma}%
  \label{lmm:sopsi}
  Given the interpolation data~\cref{eqn:data}, the transfer
  function~\cref{eqn:sopsi} of the second-order model that yields the
  interpolation conditions~\cref{eqn:interpso} can be equivalently expressed as
  \begin{equation*}
     \hH(s) = \frac{\bhc^{\trans} \bhPsi(s)^{-1} \bhb}{1 + s \bhf^{\trans}
      \bhPsi(s)^{-1} \bhb}
  \end{equation*}
  where $\bhPsi(s) = (s \bhM  - \bhK \bLambda^{-1}) (s \bI_{r} -\bLambda)$,
  $\bhf = \bLambda^{-1} \bones_{r}$ and $\bLambda$, $\bhc$, and $\bhb$ are
  defined as in \Cref{lmm:fointerp}. 
\end{lemma}
\begin{proof}
  As previously done in the proof of \Cref{lmm:sophi}, we apply \Cref{prp:smw} to the
  formulation of $\hH(s)$ in~\cref{eqn:sopsi}, with the following choice of
  vectors and matrices
  \begin{equation*}
    \begin{aligned}
      \bu & = \bhb, &
      \bv & = s \bhf = s \bLambda^{-1} \bones_{r}, &
      \bz & = \bhc, & 
      \bX & = \bhPsi(s),
    \end{aligned}
  \end{equation*}
  which yields the result of the lemma.
\end{proof}

As in \Cref{subsec:constrK}, we can assume that the mass matrix $\bhM$ to be
the identity due to \hyperref[asm:one]{Assumption~(A1.1)}.
However, this time, we additionally assume that $\bhK$ has no higher order
Jordan blocks such that we can diagonalize the stiffness matrix while
preserving the identity mass matrix
\begin{equation*}
  \begin{aligned}
    \bhK & = \diag(\kappa_{1}, \ldots, \kappa_{r}) & \text{and} &&
    \bhM & = \bI_{r} = \diag(1, \ldots, 1).
  \end{aligned}
\end{equation*}
Therefore, the matrix-valued function $\bhPsi(s)$ is a diagonal matrix for all
$s \in \C$, which can be written as
\begin{align} \label{eqn:hphiD}
  \bhPhi(s) & = (s \bhM  - \bhK \bLambda^{-1}) (s \bI_{r} -\bLambda)
    \nonumber \\
  & = \diag \Big( (s - \lambda_{1}) (s- \kappa_{1} \lambda_{1}^{-1}), \ldots,
   (s - \lambda_{r})(s - \kappa_{r} \lambda_{r}^{-1}) \Big).
\end{align}
Using this diagonal form of the matrix-valued function in \Cref{lmm:sopsi}
yields the barcycentric form for the constrained damping case.

\begin{theorem}%
  \label{thm:baryformD}
  Given interpolation points and measurements
  $\{(\lambda_{i}, h_{i})|~1 \leq i \leq r\}$, let
  \hyperref[asm:one]{Assumptions~(A1.1)}
  and~\hyperref[asm:one]{(A1.2)} as well as
  \hyperref[asm:two]{Assumption~(A2.2)} hold.
  The barycentric form of the transfer function $\hH(s)$ corresponding to
  second-order system $\hSigma_{\mathrm{SO}}: (\bhM,\bhD, \bhK, \bhb, \bhc)$ is
  given by
  \begin{equation} \label{eqn:barytfD}
    \hH(s) = \frac{\displaystyle \sum_{i = 1}^{r}
      \frac{h_{i} w_{i}}{(s - \lambda_{i})(s - \theta_{i})}}
      {\displaystyle 1 + \sum_{i = 1}^{r}
      \frac{s w_{i} \lambda_{i}^{-1}}{(s - \lambda_{i})(s - \theta_{i})}},
  \end{equation}
  with the weights $0 \neq w_{i} \in \C$ and support points
  $\theta_{i} = \kappa_{i}\lambda_{i}^{-1}$, where $\kappa_{i} \in \C$ are
  stiffness parameters, for $1 \leq i \leq r$.
  The barycentric form~\cref{eqn:barytfD} satisfies the interpolation
  conditions~\cref{eqn:interpso}.
  Moreover, the matrices of the corresponding second-order system are given as 
  \begin{equation} \label{eqn:sosyswithD}
    \begin{aligned}
      \bhM & = \bI_{r}, &
      \bhD & = \bhb \bones_{r}^{\trans} \bLambda^{-1} - \bLambda -
        \bhK \bLambda^{-1},\\
      \bhK & = \diag(\theta_{1} \lambda_{1}, \ldots,
        \theta_{r} \lambda_{r}), &
      \bhb & = \begin{bmatrix} w_{1} & \ldots & w_{r} \end{bmatrix}^{\trans}, \\
      \bhc & = \begin{bmatrix} h_{1}  & \ldots & h_{r} \end{bmatrix}^{\trans}.
    \end{aligned}
  \end{equation}
\end{theorem}
\begin{proof}
  By using the diagonal structure of $\bhPsi(s)$ in~\cref{eqn:hphiD},
  and the structure of the other components
  \begin{equation*}
    \begin{aligned}
      \bhf & = \bLambda^{-1} \bones_{r} =
        \begin{bmatrix} \lambda_{1}^{-1} & \ldots &
        \lambda_{r}^{-1}\end{bmatrix}^{\trans}, &
      \bhb & = \begin{bmatrix} w_{1} &\ldots & w_{r} \end{bmatrix}^{\trans}, &
      \bhc & = \begin{bmatrix} h_{1} & \ldots & h_{r} \end{bmatrix}^{\trans},
    \end{aligned}
  \end{equation*}
  in the formulation of the transfer function in \Cref{lmm:sopsi},
  the transfer function is rewritten in barycentric form by multipliying out
  the matrix-vector products as
  \begin{equation*}
    \hH(s) = \frac{\bhc^{\trans} \bhPsi(s)^{-1} \bhb}{1 + s \bhf^{\trans}
      \bhPsi(s)^{-1} \bhb}
    = \frac{\displaystyle \sum_{i = 1}^{r} \frac{h_{i} w_{i}}
      {(s - \lambda_{i})(s - \kappa_{i} \lambda_{i}^{-1})}}
      {\displaystyle 1 + \sum_{i = 1}^{r} \frac{s w_{i} \lambda_{i}^{-1}}
      {(s - \lambda_{i})(s - \kappa_{i} \lambda_{i}^{-1})}}.
  \end{equation*}
  Setting the support points $\theta_{i}  = \kappa_{i} \lambda_{i}^{-1}$, the
  result in~\cref{eqn:barytfD} follows directly.
  The realization~\cref{eqn:sosyswithD} is then given by rearranging the
  different parameters into the corresponding matrices and vectors.
\end{proof}

As for the realization of matrices in \Cref{thm:baryformK}, the matrices
in~\cref{eqn:sosyswithD} can be reformulated by introducing the notation
$\bTheta = \diag(\theta_{1}, \ldots, \theta_{r})$ such that
\begin{equation*}
  \begin{aligned}
    \bhM & = \bI_{r}, &
    \bhD & = \bhb \bones_{r}^{\trans} \bLambda^{-1} - \bLambda - \bTheta, \\
    \bhK & = \bTheta \bLambda, &
    \bhb & = \begin{bmatrix} w_{1}  & \ldots & w_{r}
        \end{bmatrix}^{\trans}, \\
    \bhc & = \begin{bmatrix} h_{1}  & \ldots & h_{r} \end{bmatrix}^{\trans}.
  \end{aligned}
\end{equation*}
The free parameters that explicitly appear above are $2r$ in total, and are
given by the entries of the vector $\bhb$ and of the diagonal matrix
$\bTheta$, i.e., the free parameters in the structured barycentric
form~\cref{eqn:barytfD} are  $\{w_{1}, \ldots, w_{r} \}  \cup 
\{\theta_{1}, \cdots, \theta_{r} \}$.

%%%%%%%%%%%%%%%%%%%%%%%%%%%%%%%%%%%%%%%%%%%%%%%%%%%%%%%%%%%%%%%%%%%%%%%%%%%%%%%%
% COMPUTATIONAL METHODS.                                                       %
%%%%%%%%%%%%%%%%%%%%%%%%%%%%%%%%%%%%%%%%%%%%%%%%%%%%%%%%%%%%%%%%%%%%%%%%%%%%%%%%
  
\section{Computational methods}%
\label{sec:methods}

In this section, we discuss computational aspects for the construction of
interpolating second-order models from data based on the different barycentric
forms introduced in the previous section.

%%%%%%%%%%%%%%%%%%%%%%%%%%%%%%%%%%%%%%%%%%%%%%%%%%%%%%%%%%%%%%%%%%%%%%%%%%%%%%%%

\subsection{Linearized structured barycentric Loewner frameworks}

The different barycentric forms presented in this paper are designed to
interpolate, by construction, given transfer function data
$\{(\lambda_{i}, h_{i})|~1 \leq i \leq r\}$.
However, in all three structured forms, free parameters remain.
In the first-order case (\Cref{cor:fobary}), these can be used, for example, to
interpolate additional transfer function data
$\{(\mu_{i}, g_{i})|~1 \leq i \leq r\}$, where it is assumed that the
interpolation points of the two sets are distinct, i.e.,
\begin{equation*}
  \{ \lambda_{1}, \ldots, \lambda_{r} \} \cap \{ \mu_{1}, \ldots, \mu_{r} \}
    = \emptyset.
\end{equation*}
The resulting method can be seen as a barycentric transfer function version of
the unstructured (first-order) Loewner framework~\cite{MayA07}.

\begin{algorithm}[t]
  \SetAlgoHangIndent{1pt}
  \DontPrintSemicolon
  \caption{Linearized $K$-constrained second-order barycentric Loewner
    framework.}
  \label{alg:LoewnerK}
  
  \KwIn{Left and right interpolation data
    $\{(\lambda_{i}, h_{i})|~1 \leq i \leq r\}$ and
    $\{(\mu_{i}, g_{i})|~1 \leq i \leq r\}$, support points
    $\sigma_{1}, \ldots, \sigma_{r} \in \C$.}
  \KwOut{Second-order system matrices $\bhM, \bhD, \bhK, \bhb, \bhc$.}
  
  Construct the $r$-dimensional divided differences matrix
    \begin{equation*}
      \bL_{\mathrm{K}} = \begin{bmatrix}
        \frac{\displaystyle h_{1} - g_{1}}
          {\displaystyle (\mu_{1} - \lambda_{1})(\mu_{1} - \sigma_{1})} &
        \cdots &
        \frac{\displaystyle h_{r} - g_{1}}
          {\displaystyle (\mu_{1} - \lambda_{r})(\mu_{1} - \sigma_{r})}\\
        \vdots & \ddots & \vdots \\
        \frac{\displaystyle h_{1} - g_{r}}
          {\displaystyle (\mu_{r} - \lambda_{1})(\mu_{r} - \sigma_{1})} &
        \cdots &
        \frac{\displaystyle h_{r} - g_{r}}
          {\displaystyle (\mu_{r} - \lambda_{r})(\mu_{r} - \sigma_{r})}
        \end{bmatrix}.
    \end{equation*}
    \vspace{-\baselineskip}\;
  
  Solve the linear system of equations $\bL_{\mathrm{K}} \bw = \bg$,
    for the unknown weights
    $\bw = \begin{bmatrix} w_{1} & \ldots & w_{r} \end{bmatrix}^{\trans}$ and
    the given data
    $\bg = \begin{bmatrix} g_{1} & \ldots & g_{r} \end{bmatrix}^{\trans}$.
  
  Construct the second-order system matrices
    \begin{equation*}
      \begin{aligned}
        \bhM & = \bI_{r}, &
        \bhD & = -\bLambda - \bSigma, &
        \bhK & = \bhb \bones_{r}^{\trans} + \bLambda \bSigma, &
        \bhb & = \bw, &
        \bhc & = \begin{bmatrix} h_{1}  & \ldots & h_{r} \end{bmatrix}^{\trans},
      \end{aligned}
    \end{equation*}
    with $\bLambda = \diag(\lambda_{1}, \ldots, \lambda_{r})$ and
    $\bSigma = \diag(\sigma_{1}, \ldots, \sigma_{r})$.
\end{algorithm}

Here, we aim to derive similar algorithms for the interpolation of
additional transfer function data $\{(\mu_{i}, g_{i})|~1 \leq i \leq r\}$ via
the new structured barycentric
forms~\cref{eqn:barytfK,eqn:barytfKD0,eqn:barytfD} by making use of the 
remaining free parameters.
We can observe that~\cref{eqn:barytfK,eqn:barytfD} have $2r$ free parameters
left, which potentially allow constructing of models that match the same number
of interpolation conditions.
However, the resulting systems of equations that need to be solved are nonlinear
in the unknowns and, thus, need thorough investigations in terms of solvability.
For simplicity of exposition, we consider in this work only linearized versions
of the equations by choosing ``suitable'' $\sigma_{i}$'s in~\cref{eqn:barytfK}
and $\theta_{i}$'s in~\cref{eqn:barytfD} a priori, leading to small linear
systems to solve to satisfy additional $r$ interpolation conditions.
A discussion of heuristic choices for these support points is given in
the upcoming \Cref{subsec:linparam}.

\begin{algorithm}[t]
  \SetAlgoHangIndent{1pt}
  \DontPrintSemicolon
  \caption{Linearized $D$-constrained second-order barycentric Loewner
    framework.}
  \label{alg:LoewnerD}
  
  \KwIn{Left and right interpolation data
    $\{(\lambda_{i}, h_{i})|~1 \leq i \leq r\}$ and
    $\{(\mu_{i}, g_{i})|~1 \leq i \leq r\}$, support points
    $\theta_{1}, \ldots,  \theta_{r} \in \C$.}
  \KwOut{Second-order system matrices $\bhM, \bhD, \bhK, \bhb, \bhc$.}
  Construct the $r$-dimensional divided differences matrix
    \begin{equation*}
      \bL_{\mathrm{D}} = \begin{bmatrix}
        \frac{\displaystyle h_{1} - \mu_{1} g_{1} \lambda_{1}^{-1}}
          {\displaystyle (\mu_{1} - \lambda_{1})(\mu_{1} - \theta_{1})} &
        \cdots &
        \frac{\displaystyle h_{r} - \mu_{1} g_{1} \lambda_{r}^{-1}}
          {\displaystyle (\mu_{1} - \lambda_{r})(\mu_{1} - \theta_{r})}\\
        \vdots & \ddots & \vdots \\
        \frac{\displaystyle h_{1} - \mu_{r} g_{r} \lambda_{1}^{-1}}
          {\displaystyle (\mu_{r} - \lambda_{1})(\mu_{r} - \theta_{1})} &
        \cdots &
        \frac{\displaystyle h_{r} - \mu_{r} g_{r} \lambda_{r}^{-1}}
          {\displaystyle (\mu_{r} - \lambda_{r})(\mu_{r} - \theta_{r})}
        \end{bmatrix}.
    \end{equation*}
    \vspace{-\baselineskip}\;
  
  Solve the linear system of equations $\bL_{\mathrm{D}} \bw = \bg$,
    for the unknown weights
    $\bw = \begin{bmatrix} w_{1} & \ldots & w_{r} \end{bmatrix}^{\trans}$
    and the given data
    $\bg = \begin{bmatrix} g_{1} & \ldots & g_{r} \end{bmatrix}^{\trans}$.
  
  Construct the second-order system matrices
    \begin{equation*}
      \begin{aligned}
        \bhM & = \bI_{r}, &
        \bhD & = \bhb \bhf^{\trans} - \bLambda - \Theta, &
        \bhK & =  \bTheta \bLambda, &
        \bhb & = \bw, &
        \bhc & = \begin{bmatrix} h_{1}  & \ldots & h_{r} \end{bmatrix}^{\trans},
      \end{aligned}
    \end{equation*}
    with $\bhf^{\trans} = \begin{bmatrix} \lambda_{1}^{-1} & \ldots &
    \lambda_{r}^{-1} \end{bmatrix}$,
    $\bLambda = \diag(\lambda_{1}, \ldots, \lambda_{r})$ and
    $\bTheta = \diag(\theta_{1}, \ldots, \theta_{r})$.
\end{algorithm}

The resulting methods based on the barycentric
forms~\cref{eqn:barytfK,eqn:barytfD} are given in
\Cref{alg:LoewnerK,alg:LoewnerD}.
Both algorithms follow a similar structure.
In Step~1 of \Cref{alg:LoewnerK,alg:LoewnerD}, the given interpolation data and
support points are used to set up divided differences matrices.
The realizations of these matrices are determined by the underlying barycentric
forms~\cref{eqn:barytfK,eqn:barytfD}.

For example, consider the barycentric form in~\cref{eqn:barytfK}.
The goal is to find the weights $\bw = \begin{bmatrix} w_{1} & \ldots & w_{r}
\end{bmatrix}^{\trans}$ in this barycentric form~\cref{eqn:barytfK} such
that the additional interpolation conditions~\cref{eqn:cond2} for
$\{(\mu_{i}, g_{i})|~1 \leq i \leq r\}$ are satisfied.
By inserting the form~\cref{eqn:barytfK} into the interpolation
conditions~\cref{eqn:cond2} and by multiplying both sides with the denominator
of the barycentric form, we obtain the new relation
\begin{equation*}
  \sum\limits_{i = 1}^{r} \frac{h_{i} w_{i}}
    {(\mu_{j} - \lambda_{i})(\mu_{j} - \sigma_{i})}
    = g_{j} + \sum\limits_{i = 1}^{r} \frac{g_{j} w_{i}}
    {(\mu_{j} - \lambda_{i})(\mu_{j} - \sigma_{i})},
\end{equation*}
for $j = 1, \ldots, r$.
Bringing the terms with the unknowns $w_{1}, \ldots, w_{r}$ to the left-hand
side yields the relation
\begin{equation} \label{eqn:linsys}
  \sum\limits_{i = 1}^{r} \frac{(h_{i} - g_{j}) w_{i}}
    {(\mu_{j} - \lambda_{i})(\mu_{j} - \sigma_{i})} = g_{j},
\end{equation}
for $j = 1, \ldots, r$.
Rearranging all equations of the form~\cref{eqn:linsys} such that the unknowns
can be written as the vector $\bw = \begin{bmatrix} w_{1} & \ldots & w_{r}
\end{bmatrix}^{\trans}$ results in the linear system of equations
\begin{equation*}
  \bL_{\mathrm{K}} \bw = \bg,
\end{equation*}
with the data vector $\bg = \begin{bmatrix} g_{1} & \ldots & g_{r}
\end{bmatrix}^{\trans}$ and the matrix of divided differences
\begin{equation*}
  \bL_{\mathrm{K}} = \begin{bmatrix}
    \frac{\displaystyle h_{1} - g_{1}}
      {\displaystyle (\mu_{1} - \lambda_{1})(\mu_{1} - \sigma_{1})} &
    \cdots &
    \frac{\displaystyle h_{r} - g_{1}}
      {\displaystyle (\mu_{1} - \lambda_{r})(\mu_{1} - \sigma_{r})}\\
    \vdots & \ddots & \vdots \\
    \frac{\displaystyle h_{1} - g_{r}}
      {\displaystyle (\mu_{r} - \lambda_{1})(\mu_{r} - \sigma_{1})} &
    \cdots &
    \frac{\displaystyle h_{r} - g_{r}}
    {\displaystyle (\mu_{r} - \lambda_{r})(\mu_{r} - \sigma_{r})}
  \end{bmatrix}.
\end{equation*}
This system of linear equations is then solved in Step~2 of
\Cref{alg:LoewnerK}.
A similar derivation using~\cref{eqn:barytfD} leads to the divided differences
matrix in Step~1 of \Cref{alg:LoewnerD} and the solve of an analogous linear
system of equations in Step~2 of \Cref{alg:LoewnerD}.
Under the assumption that the number of given data points $r$ is less than
the minimal system dimension and suitable choices of support points
$\{\sigma_{i}\}$ and $\{\theta_{i}\}$ such that
\hyperref[asm:one]{Assumptions~(A1.2)}, \hyperref[asm:two]{~(A2.1)}
and~\hyperref[asm:two]{(A2.2)} are satisfied for all interpolation points
$\lambda_{1}, \ldots, \lambda_{r}, \mu_{1}, \ldots, \mu_{r}$,
these linear systems of equations have unique solutions.

In Step~3 of~\Cref{alg:LoewnerK,alg:LoewnerD}, the second-order systems are
constructed following the theory of \Cref{thm:baryformK,thm:baryformD}.
In both cases, the transfer functions of the constructed systems satisfy the
$2r$ imposed interpolation conditions
\begin{equation} \label{eqn:cond2}
  \begin{aligned}
    \hH(\lambda_{1}) & = h_{1}, & \ldots, && \hH(\lambda_{r}) & = h_{r}, &
    \hH(\mu_{1}) & = g_{1}, & \ldots, && \hH(\mu_{r}) & = g_{r}.
  \end{aligned}
\end{equation}

\begin{algorithm}[t]
  \SetAlgoHangIndent{1pt}
  \DontPrintSemicolon
  \caption{Second-order barycentric Loewner framework with zero damping.}
  \label{alg:LoewnerKD0}
  
  \KwIn{Left and right interpolation data
    $\{(\lambda_{i}, h_{i})|~1 \leq i \leq r\}$ and
    $\{(\mu_{i}, g_{i})|~1 \leq i \leq r\}$.}
  \KwOut{Second-order system matrices $\bhM, \bhD, \bhK, \bhb, \bhc$.}
  
  Construct the $r$-dimensional divided difference matrix
    \begin{equation*}
      \bL_{\mathrm{KD0}} = \begin{bmatrix}
        \frac{\displaystyle h_{1} - g_{1}}
          {\displaystyle \mu_{1}^{2} - \lambda_{1}^{2}} &
        \cdots &
        \frac{\displaystyle h_{r} - g_{1}}
          {\displaystyle \mu_{1}^{2} - \lambda_{r}^{2}}\\
        \vdots & \ddots & \vdots \\
        \frac{\displaystyle h_{1} - g_{r}}
          {\displaystyle \mu_{r}^{2} - \lambda_{1}^{2}} &
        \cdots &
        \frac{\displaystyle h_{r} - g_{r}}
          {\displaystyle \mu_{r}^{2} - \lambda_{r}^{2}}
        \end{bmatrix}.
    \end{equation*}
    \vspace{-\baselineskip}\;
  
  Solve the linear system of equations $\bL_{\mathrm{KD0}} \bw = \bg$,
    for the unknown weights
    $\bw = \begin{bmatrix} w_{1} & \ldots & w_{r} \end{bmatrix}^{\trans}$
    and the given data
    $\bg = \begin{bmatrix} g_{1} & \ldots & g_{r} \end{bmatrix}^{\trans}$.
  
  Construct the second-order system matrices
    \begin{equation*}
      \begin{aligned}
        \bhM & = \bI_{r}, &
        \bhD & = 0, &
        \bhK & = \bhb \bones_{r}^{\trans} - \bLambda^{2}, &
        \bhb & = \bw, &
        \bhc & = \begin{bmatrix} h_{1}  & \ldots & h_{r} \end{bmatrix}^{\trans},
      \end{aligned}
    \end{equation*}
    with $\bLambda = \diag(\lambda_{1}, \ldots, \lambda_{r})$.
\end{algorithm}

\Cref{alg:LoewnerKD0} shows the barycentric Loewner framework for the case of
zero damping matrices based on \Cref{cor:barytfKD0}.
While the main computational steps are following the same ideas as in
\Cref{alg:LoewnerK,alg:LoewnerD}, the difference to these methods is the lack
of the set of support points $\{\sigma_{i}\}_{i = 1, \ldots r}$ and
$\{\theta_{i}\}_{i = 1, \ldots r}$, which results from enforcing the $\bhD = 0$
damping model.
The corresponding barycentric form~\cref{eqn:barytfKD0} has $r$ remaining free
parameters that exactly allow the construction of an interpolating second-order
system satisfying~\cref{eqn:cond2}.

%%%%%%%%%%%%%%%%%%%%%%%%%%%%%%%%%%%%%%%%%%%%%%%%%%%%%%%%%%%%%%%%%%%%%%%%%%%%%%%%

\subsection{Construction of systems with real-valued matrices}%
\label{subsec:realmodel}

A property desired in many applications for learned systems is the realization
by means of real-valued matrices.
This often allows the reinterpretation of the learned quantities and the use
of classical tools, established for real systems resulting, for example, from
finite element discretizations.
The key feature of second-order systems~\cref{eqn:sosys} with real matrices
that needs to be exploited for the construction is that data at complex
conjugate frequency points are also complex conjugate:
\begin{equation*}
  H(\overline{s}) = \bc^{\trans} (\overline{s}^{2} \bM + \overline{s} \bD +
    \bK)^{-1} \bb
    = \overline{\bc^{\trans} (s^{2} \bM + s \bD + \bK)^{-1} \bb}
    = \overline{H(s)}.
\end{equation*}
As in the classical Loewner framework, we assume the given data
$\{(\lambda_{i}, h_{i})|~1 \leq i \leq r\}$ and
$\{(\mu_{i}, g_{i})|~1 \leq i \leq r\}$ to be closed under conjugation in the
respective sets.
Additionally, for \Cref{alg:LoewnerK,alg:LoewnerD}, we need
to assume that the support points $\{\sigma_{i}\}_{i = 1, \ldots r}$
and $\{\theta_{i}\}_{i = 1, \ldots r}$ are also closed under conjugation
and that if $\lambda_{i}, \lambda_{i+1}$ are complex conjugate, then so are
$\sigma_{i}, \sigma_{i+1}$ or $\theta_{i}, \theta_{i+1}$, respectively.
Let the interpolation data and parameters be ordered such that complex conjugate
are sorted together, e.g., for the interpolation points in
$\{(\lambda_{i}, h_{i})|~1 \leq i \leq r\}$, we have that
\begin{equation*}
  \lambda_{1},~ \lambda_{2} = \overline{\lambda_{1}},~
  \lambda_{3},~ \lambda_{4} = \overline{\lambda_{3}},~ \ldots.
\end{equation*}
Given the matrices $\bhM, \bhD, \bhK, \bhb, \bhc$ computed by any of the
\Cref{alg:LoewnerK,alg:LoewnerD,alg:LoewnerKD0}, a real-valued realization of
the described system is given by
\begin{equation*}
  \begin{aligned}
    \overline{P}^{\trans} \bhM \overline{P}, &&
    \overline{P}^{\trans} \bhD \overline{P}, &&
    \overline{P}^{\trans} \bhK \overline{P}, &&
    \overline{P}^{\trans} \bhb, &&
    \bhc \overline{P},
  \end{aligned}
\end{equation*}
where the transformation matrix $P$ is block diagonal with
\begin{equation*}
  P = \diag(J_{1}, J_{2}, \ldots, J_{\ell}) \in \C^{r \times r},
\end{equation*}
and the block matrices are chosen according to the given interpolation
data by
\begin{equation*}
  J_{k} = \begin{cases}
    \frac{1}{\sqrt{2}}
    \begin{bmatrix} 1 & -\i \\ 1 & \i \end{bmatrix} &
    \text{for complex conjugate interpolation points,}\\
    1 & \text{for real interpolation points.}
  \end{cases}
\end{equation*}

\begin{remark}
  A special situation occurs in the case of \Cref{alg:LoewnerKD0} for the
  construction of models with zero damping matrix.
  As discussed at the end of \Cref{subsubsec:constrKD0}, the collection
  of data on the imaginary axis $\i \R$ in complex conjugate pairs violates
  \hyperref[asm:two]{Assumption~(A2.1)}.
  This is consistent with the observation that the corresponding transfer
  function yields identical values for complex conjugate points on
  the imaginary axis, i.e., for any $\bM, \bK \in \C^{n \times n}$, $\bD = 0$
  and $\bb, \bc \in \C^{n}$, it holds that
  \begin{equation*}
    H(s) = \bc^{\trans} (s^{2} \bM + \bK)^{-1} \bb = H(\overline{s}),
  \end{equation*}
  for all $s = \i \omega \in \i \R$.
  As a result, the collection of data from complex conjugate points on the
  imaginary axis yields no additional information due to the specific system
  structure and leads to singular linear systems in Step~2 of
  \Cref{alg:LoewnerKD0}.
  A side effect is that systems with zero damping and real matrices produce only
  real data on the imaginary axis such that no additional enforcement of real
  valued matrices is necessary.
\end{remark}

%%%%%%%%%%%%%%%%%%%%%%%%%%%%%%%%%%%%%%%%%%%%%%%%%%%%%%%%%%%%%%%%%%%%%%%%%%%%%%%%

\subsection{Heuristics for choosing the support points}%
\label{subsec:linparam}

To discuss suitable choices for the support points
$\sigma_{1}, \ldots, \sigma_{r}$ and $\theta_{1}, \ldots, \theta_{r}$, we
consider their influence on the transfer
functions~\cref{eqn:barytfK,eqn:barytfD}.
First, consider the case of the barycentric form~\cref{eqn:barytfK} resulting
from the constrained stiffness matrix.
Assume for the moment that $\sigma_{1}, \ldots, \sigma_{r}$ are all distinct.
Then, for the transfer function~\cref{eqn:barytfK}, it holds that
\begin{equation*}
  \begin{aligned}
    \hH(\sigma_{i}) & = \frac{h_{i} w_{i}}{w_{i}} = h_{i}, &
    \text{for all}~1 \leq i \leq r,
  \end{aligned}
\end{equation*}
i.e., the transfer function assumes the same values at the support points
as at the interpolation points.
Since the typical case will be $H(\lambda_{i}) \neq H(\sigma_{i})$, this
introduces approximation errors at the chosen support points.
In the case that some of the support points are chosen to be identical, i.e.,
$\sigma_{i} = \sigma_{i_{1}} = \ldots = \sigma_{i_{\ell}}$ for some
indices $i_{1}, \ldots, i_{\ell}$, one can observe that
\begin{equation} \label{eqn:sptpts1}
  \begin{aligned}
    \hH(\sigma_{i}) & = \frac{\sum\limits_{k = 1}^{\ell}
      h_{i_{k}} w_{i_{k}} (\sigma_{i} - \lambda_{i_{k}})}
      {\sum\limits_{k = 1}^{\ell}w_{i_{k}} (\sigma_{i} - \lambda_{i_{k}})}, &
      \text{for all}~1 \leq i \leq r,
  \end{aligned}
\end{equation}
holds.
Similar to the case of distinct support points, approximation errors at 
$\sigma_{i}$ are introduced.
However, choosing many support points to be identical, allows to
cluster the introduced errors away from frequency ranges of interest.

In a similar way, one can observe for the barycentric
form~\cref{eqn:barytfD} resulting from the constrained damping matrix that in
the case of distinct support points $\theta_{1}, \ldots, \theta_{r}$, it holds
\begin{equation*} 
  \begin{aligned}
    \hH(\theta_{i}) & = \frac{h_{i} w_{i}}{w_{i}\theta_{i} \lambda_{i}^{-1}}
      = \frac{h_{i} \lambda_{i}}{\theta_{i}}, &
      \text{for all}~1 \leq i \leq r.
  \end{aligned}
\end{equation*}
As before, typically $H(\theta_{i}) \neq \frac{h_{i} \lambda_{i}}{\theta_{i}}$
will hold, indicating approximation errors introduced at the chosen
support points.
In the case that some of the support points are identical, i.e.,
$\theta_{i} = \theta_{i_{1}} = \ldots = \theta_{i_{\ell}}$ for some
indices $i_{1}, \ldots, i_{\ell}$, it holds that
\begin{equation} \label{eqn:sptpts2}
  \begin{aligned} 
    \hH(\theta_{i}) & = \frac{\sum\limits_{k = 1}^{\ell}
      h_{i_{k}} w_{i_{k}} (\theta_{i} - \lambda_{i_{k}})}
      {\sum\limits_{k = 1}^{\ell} \theta_{i} w_{i_{k}} \lambda_{i_{k}}^{-1}
      (\theta_{i} - \lambda_{i_{k}})}, &
      \text{for all}~1 \leq i \leq r.
  \end{aligned}
\end{equation}

To summarize, \Cref{eqn:sptpts1,eqn:sptpts2} show that poorly chosen support
points introduce undesired approximation errors.
The points $\sigma_{1}, \ldots, \sigma_{r}$ and $\theta_{1}, \ldots, \theta_{r}$
do not need to be distinct (in contrast to the interpolation points; cf.
\hyperref[asm:one]{Assumption~(A1.2)}), and the undesired approximation errors
are enforced at the support points themselves.
The second observation motivates to choose $\sigma_{1}, \ldots, \sigma_{r}$
and $\theta_{1}, \ldots, \theta_{r}$ outside the considered frequency region of
interest, i.e.,
\begin{equation*}
  \sigma_{k}, \theta_{k}
  \not\in \conv_{\R}\{ \lambda_{1}, \ldots, \lambda_{r},
  \mu_{1}, \ldots, \mu_{r} \},
\end{equation*}
for $k = 1, \ldots, r$ and where
\begin{equation*}
  \conv_{\R}\{a_{1}, \ldots, a_{\ell}\} :=
    \left\{ z \in \C \left|~ z = \sum_{k = 1}^{\ell} \beta_{k} a_{k},~
    \sum_{k = 1}^{\ell} \beta_{k} = 1,~
    \beta_{1} \geq 0, \ldots, \beta_{\ell} \geq 0 \right. \right\}
\end{equation*}
denotes the convex hull of elements in $\C$ over $\R$.
This can be achieved, for example, by taking large shifts or multiples of
the given interpolation points for the support points.

Next, we revisit the construction of the final
second-order system matrices in \Cref{alg:LoewnerK,alg:LoewnerD} and the
influence of the support points on the properties of these matrices.
In both algorithms, the damping matrices take into account the negatives of the
interpolation points $\lambda_{1}, \ldots, \lambda_{r}$, which are subtracted
by the chosen support points $\sigma_{1}, \ldots, \sigma_{r}$ or
$\theta_{1}, \ldots, \theta_{r}$.
Eigenvalues with positive real parts in the damping matrix can be interpreted as
dissipation of energy from the system.
To achieve this, a reasonable choice for $\sigma_{1}, \ldots, \sigma_{r}$ and
$\theta_{1}, \ldots, \theta_{r}$ is to have negative real parts,
which pushes the eigenvalues of the resulting damping matrix toward the right
open half-plane.
Especially, it is possible to construct real, symmetric positive definite
damping matrices via \Cref{alg:LoewnerK} in the case that the interpolation data
has been obtained on the imaginary axis by choosing the support points
$\sigma_{1}, \ldots, \sigma_{r}$ to yield
\begin{equation} \label{eqn:choice1}
  \begin{aligned}
    \real(\sigma_{i}) & < 0 & \text{and} &&
    \imag(\sigma_{i}) & = -\imag(\lambda_{i}), &
    \text{for all}~i = 1, \ldots, r.
  \end{aligned}
\end{equation}

On the other hand, we observe that the stiffness matrices in
\Cref{alg:LoewnerK,alg:LoewnerD} involve the multiplication of the interpolation
points $\lambda_{1}, \ldots, \lambda_{r}$ with the support points
$\sigma_{1}, \ldots, \sigma_{r}$ and $\theta_{1}, \ldots, \theta_{r}$.
Aiming for stiffness matrices that have eigenvalues with positive real parts,
which together with eigenvalues with positive real parts in the damping
matrix drives the system towards asymptotic stability, the support points
$\sigma_{1}, \ldots, \sigma_{r}$ and $\theta_{1}, \ldots, \theta_{r}$
should be chosen to have imaginary parts in the opposite half-plane of
the imaginary parts of $\lambda_{1}, \ldots, \lambda_{r}$, e.g.,
if the interpolation points are chosen over the positive imaginary axis, then
$\sigma_{1}, \ldots, \sigma_{r}$ and $\theta_{1}, \ldots, \theta_{r}$ should
have negative imaginary parts.
In the case of \Cref{alg:LoewnerD} and the interpolation points
$\lambda_{1}, \ldots, \lambda_{r}$ to be on the imaginary axis, the 
support points $\theta_{1}, \ldots, \theta_{r}$ can be chosen such that
$\bhK$ can be realized as a real-valued, symmetric positive definite matrix by
\begin{equation*}
  \begin{aligned}
    \imag(\theta{i}) & = -\left( \imag(\lambda_{i}) + c_{i} \right),
    & \text{for all}~i = 1, \ldots, r,
  \end{aligned}
\end{equation*}
where $c_{i} \in \R$ are real constants with the same sign as the imaginary
part of $\lambda_{i}$ such that $c_{i} \imag(\lambda_{i})) > 0$ holds.

%%%%%%%%%%%%%%%%%%%%%%%%%%%%%%%%%%%%%%%%%%%%%%%%%%%%%%%%%%%%%%%%%%%%%%%%%%%%%%%%
% NUMERICAL EXAMPLES.                                                          %
%%%%%%%%%%%%%%%%%%%%%%%%%%%%%%%%%%%%%%%%%%%%%%%%%%%%%%%%%%%%%%%%%%%%%%%%%%%%%%%%

\section{Numerical experiments}%
\label{sec:examples}

In this section, we verify the proposed algorithms and barycentric forms
numerically in different examples.
The experiments reported here have been executed on a machine equipped with an
AMD Ryzen 5 5500U processor running at 2.10\,GHz and with
16\,GB total main memory.
The computer runs on Windows 10 Home version 21H2 (build 19044.2251) and, for
all reported experiments, we use \matlab{} 9.9.0.1592791 (R2020b).
Source codes, data and numerical results are available at~\cite{supWer23a}.

%%%%%%%%%%%%%%%%%%%%%%%%%%%%%%%%%%%%%%%%%%%%%%%%%%%%%%%%%%%%%%%%%%%%%%%%%%%%%%%%

\subsection{Computational setup}

As setup of the subsequent comparative study, we consider the following data
driven, interpolation-based methods:
\begin{description}
  \item[\soloewnerk{}] the linearized $K$-constrained second-order barycentric
    Loewner framework from \Cref{alg:LoewnerK},
  \item[\soloewnerd{}] the linearized $D$-constrained second-order barycentric
    Loewner framework from \Cref{alg:LoewnerD},
  \item[\soloewnerkdo{}] the second-order barycentric Loewner framework with
    zero damping matrix from \Cref{alg:LoewnerKD0},
  \item[\soloewnerr{}] the second-order (matrix) Loewner framework for
    Rayleigh-damped systems from~\cite{PonGB22},
  \item[\loewner{}] the classical first-order barycentric Loewner framework
    based on the results in \Cref{cor:fobary}; see also~\cite{AntA86}.
\end{description}
All models are constructed such that the transfer functions satisfy the same
interpolation conditions.

As interpolation points, we have chosen the local minima and maxima of the given
data samples on the positive part of the imaginary axis using the \matlab{}
functions \texttt{islocalmin} and \texttt{islocalmax} supplemented by the limits
of the considered frequency intervals of interest
$[\i \omega_{\min}, \i \omega_{\max}]$ and,
if necessary, some additional intermediate points.
The interpolation points are then split into the left and right sets by
alternating the ascending order of the positive imaginary parts, i.e.,
\begin{equation*}
  \imag(\lambda_{1}) < \imag(\mu_{1}) <
  \imag(\lambda_{2}) < \imag(\mu_{2}) <\ldots <
  \imag(\lambda_{r}) < \imag(\mu_{r}).
\end{equation*}
For \soloewnerk{} and \soloewnerd{}, the support points are chosen
according to the discussion in \Cref{subsec:linparam}, more specifically as
in~\cref{eqn:choice1}.
The individual choices are outlined in the description of the examples below.
The Rayleigh damping parameters for \soloewnerr{} are either known from the
original model or inferred in an optimal way.
This means that we do not consider the optimization process of these additional
parameters here.

For the comparison of the different methods, we consider the transfer function
magnitude, given as the absolute value
$\lvert H(\i \omega_{k}) \rvert$, in the discrete frequency points
$\i \omega_{k}$ given in the data sets, and the corresponding pointwise relative
errors via
\begin{equation*}
  \relerr(\omega_{k}) :=
    \frac{\lvert H(\i \omega_{k}) - \hH(\i \omega_{k}) \rvert}
    {\lvert H(\i \omega_{k}) \rvert}.
\end{equation*}

\begin{table}[t]
  \centering
  \caption{Dimensions and numbers of data samples for numerical examples.
    The amount of matched interpolation conditions for all constructed
    models is $2r$, due to the partition into left and right data.
    The column ``c.c.d'' denotes the use of complex conjugate interpolation
    points and data.}
  \label{tab:dim}
  
  \vspace{\baselineskip}
  {\renewcommand{\arraystretch}{1.25}%
  \setlength{\tabcolsep}{.8em}
  \begin{tabular}{rrrr}
    \hline
    & 
      \multicolumn{1}{c}{\textbf{full order}} &
      \multicolumn{1}{c}{$\boldsymbol{r}$} &
      \multicolumn{1}{c}{\textbf{c.c.d.}}\\
    \hline\noalign{\medskip}
    Butterfly gyroscope &
      $17\,361$ &
      $14$ &
      \texttt{true}\\
    Artificial fishtail &
      $779\,232$ &
      $6$ &
      \texttt{true}\\
    Flexible aircraft &
      --- &
      $52$ &
      \texttt{true}\\
    \noalign{\smallskip}\hline\noalign{\smallskip}
    Bone model &
      $986\,703$ &
      $19$ &
      \texttt{false} \\
    Hysteretic plate & 
      $201\,900$ &
      $52$ &
      \texttt{false}\\
    \noalign{\medskip}\hline\noalign{\smallskip}
  \end{tabular}}
\end{table}

An overview providing the dimensions of the original full-order systems, the
dimensions of the learned (reduced-order) models and if complex conjugate
data has been used in the computations for the construction of real-valued
system matrices can be found in \Cref{tab:dim}.
The rows of the table are split into the examples with and without damping
matrix.

%%%%%%%%%%%%%%%%%%%%%%%%%%%%%%%%%%%%%%%%%%%%%%%%%%%%%%%%%%%%%%%%%%%%%%%%%%%%%%%%

\subsection{Examples with non-zero damping matrix}

\begin{figure}[t]
  \centering
  \begin{subfigure}[b]{.49\textwidth}
    \centering
  \tikzexternalenable%
  \tikzsetnextfilename{butterfly_tf}%
  \begin{tikzpicture}
  \pgfplotstableread{graphics/data/butterfly_interpts.dat}\tablePTS
  \pgfplotstableread{graphics/data/butterfly_tf.dat}\tableTF
  
  \begin{loglogaxis}[%
    width  = .675\textwidth,
    height = .45\textwidth,
    scale only axis,
    xmin = 1e+2,
    xmax = 1e+6,
    ymin = 1e-8,
    ymax = 5e-2,
    xminorticks = false,
    yminorticks = false,
    xlabel = {frequency $\omega$ (rad/s)},
    ylabel = {magnitude $\lvert H(\i \omega) \rvert$},
    ylabel style   = {yshift = -.3em},
    scaled x ticks = false,
    x tick label style = {/pgf/number format/fixed},
    cycle list name    = plotlist,
    clip mode          = individual]
    
    \pgfplotsset{cycle list shift = 0}
    \addplot table[x index = 0, y index = 1] {\tablePTS};
    
    \pgfplotsset{cycle list shift = -1}
    \addplot table[x index = 2, y index = 3] {\tablePTS};
    
    \foreach \y in {2, 3, 4, 5} {
      \addplot table[x index = 0, y index = \y] {\tableTF};
    }
  \end{loglogaxis}
\end{tikzpicture}%
  \tikzexternaldisable%

    \caption{Transfer function values.}
    \label{fig:butterfly_tf}
  \end{subfigure}%
  \hfill%
  \begin{subfigure}[b]{.49\textwidth}
    \centering
  \tikzexternalenable%
  \tikzsetnextfilename{butterfly_err}%
  \begin{tikzpicture}
  \pgfplotstableread{graphics/data/butterfly_relerr.dat}\tableERR
  
  \begin{loglogaxis}[%
    width  = .675\textwidth,
    height = .45\textwidth,
    scale only axis,
    xmin = 1e+2,
    xmax = 1e+6,
    ymin = 1e-18,
    ymax = 5e-0,
    xminorticks = false,
    yminorticks = false,
    xlabel = {frequency $\omega$ (rad/s)},
    ylabel = {relative error $\relerr(\omega)$},
    ylabel style   = {yshift = -.3em},
    scaled x ticks = false,
    x tick label style = {/pgf/number format/fixed},
    cycle list name    = plotlist,
    clip mode          = individual]
    
    \pgfplotsset{cycle list shift = 1}
    \foreach \y in {1, 2, 3, 4} {
      \addplot table[x index = 0, y index = \y] {\tableERR};
    }
  \end{loglogaxis}
\end{tikzpicture}%
  \tikzexternaldisable%

    \caption{Pointwise relative errors.}
    \label{fig:butterfly_err}
  \end{subfigure}
  
  \vspace{.5\baselineskip}
  \tikzexternalenable%
  \tikzsetnextfilename{butterfly_legend}%
  \begin{tikzpicture}
  \begin{axis}[%
    hide axis,
    width  = 1mm,
    height = 1mm,
    scale only axis,
    xmin = 0,
    xmax = 1,
    ymin = 0,
    ymax = 1,
    legend columns = 3, 
    legend style   = {
      at     = {(0,0)},
      anchor = center,
      /tikz/every even column/.append style = {column sep = 0.4cm}},
    legend cell align  = {left},
    cycle list name    = plotlist,
    clip mode          = individual]
    
    \pgfplotsinvokeforeach{1, ..., 3}{\addplot coordinates {(0,0)};}
    \addlegendentry{Interpolation data}
    \addlegendentry{\soloewnerk{}}
    \addlegendentry{\soloewnerd{}}
    
    \addlegendimage{empty legend}
    \addlegendentry{}
    \pgfplotsinvokeforeach{1, 2}{\addplot coordinates {(0,0)};}
    \addlegendentry{\soloewnerr{}}
    \addlegendentry{\loewner{}}
  \end{axis}
\end{tikzpicture}%
  \tikzexternaldisable%

  \caption{Butterfly gyroscope example: All methods recover reduced-order
    models with similar accuracy using the same amount of given interpolation
    data.
    The model learned by \soloewnerr{} performs slightly better for frequencies
    between $500$ and $5\,000$\,rad/s.}
  \label{fig:butterfly}
\end{figure}
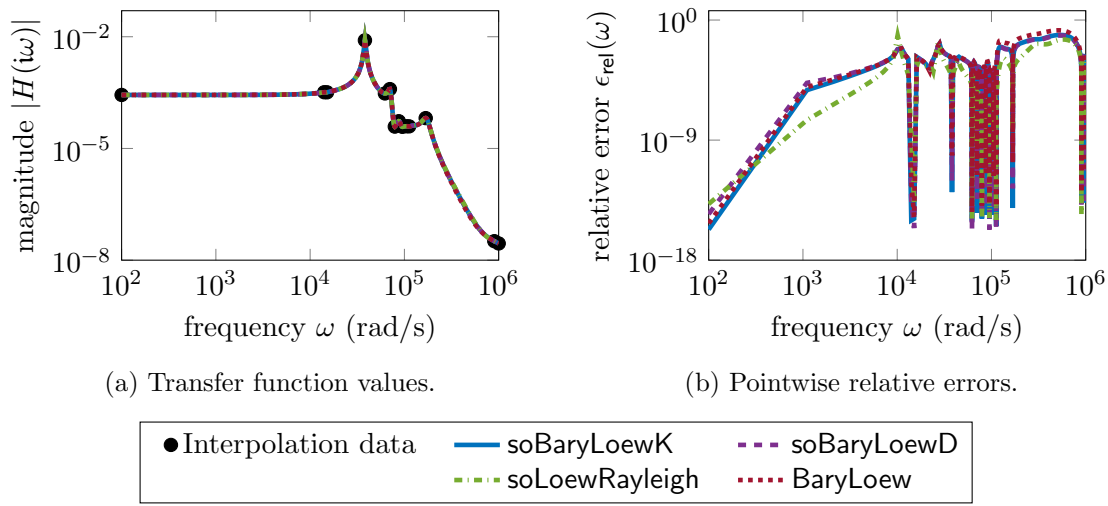

\begin{figure}[t]
  \centering
  \begin{subfigure}[b]{.49\textwidth}
    \centering
  \tikzexternalenable%
  \tikzsetnextfilename{fishtail_tf}%
  \begin{tikzpicture}
  \pgfplotstableread{graphics/data/fishtail_interpts.dat}\tablePTS
  \pgfplotstableread{graphics/data/fishtail_tf.dat}\tableTF
  
  \begin{loglogaxis}[%
    width  = .675\textwidth,
    height = .45\textwidth,
    scale only axis,
    xmin = 1e+0,
    xmax = 1e+3,
    ymin = 1e-8,
    ymax = 1e-3,
    xminorticks = false,
    yminorticks = false,
    xlabel = {frequency $\omega$ (rad/s)},
    ylabel = {magnitude $\lvert H(\i \omega) \rvert$},
    ylabel style   = {yshift = -.3em},
    scaled x ticks = false,
    x tick label style = {/pgf/number format/fixed},
    cycle list name    = plotlist,
    clip mode          = individual]
    
    \pgfplotsset{cycle list shift = 0}
    \addplot table[x index = 0, y index = 1] {\tablePTS};
    
    \pgfplotsset{cycle list shift = -1}
    \addplot table[x index = 2, y index = 3] {\tablePTS};
    
    \foreach \y in {2, 3, 4, 5} {
      \addplot table[x index = 0, y index = \y] {\tableTF};
    }
  \end{loglogaxis}
\end{tikzpicture}%
  \tikzexternaldisable%

    \caption{Transfer function values.}
    \label{fig:fishtail_tf}
  \end{subfigure}%
  \hfill%
  \begin{subfigure}[b]{.49\textwidth}
    \centering
  \tikzexternalenable%
  \tikzsetnextfilename{fishtail_err}%
  \begin{tikzpicture}
  \pgfplotstableread{graphics/data/fishtail_relerr.dat}\tableERR
  
  \begin{loglogaxis}[%
    width  = .675\textwidth,
    height = .45\textwidth,
    scale only axis,
    xmin = 1e+0,
    xmax = 1e+3,
    ymin = 1e-16,
    ymax = 5e-0,
    xminorticks = false,
    yminorticks = false,
    xlabel = {frequency $\omega$ (rad/s)},
    ylabel = {relative error $\relerr(\omega)$},
    ylabel style   = {yshift = -.3em},
    scaled x ticks = false,
    x tick label style = {/pgf/number format/fixed},
    cycle list name    = plotlist,
    clip mode          = individual]
    
    \pgfplotsset{cycle list shift = 1}
    \foreach \y in {1, 2, 3, 4} {
      \addplot table[x index = 0, y index = \y] {\tableERR};
    }
  \end{loglogaxis}
\end{tikzpicture}%
  \tikzexternaldisable%

    \caption{Pointwise relative errors.}
    \label{fig:fishtail_err}
  \end{subfigure}
  
  \vspace{.5\baselineskip}
  \tikzexternalenable%
  \tikzsetnextfilename{fishtail_legend}%
  \begin{tikzpicture}
  \begin{axis}[%
    hide axis,
    width  = 1mm,
    height = 1mm,
    scale only axis,
    xmin = 0,
    xmax = 1,
    ymin = 0,
    ymax = 1,
    legend columns = 3, 
    legend style   = {
      at     = {(0,0)},
      anchor = center,
      /tikz/every even column/.append style = {column sep = 0.4cm}},
    legend cell align  = {left},
    cycle list name    = plotlist,
    clip mode          = individual]
    
    \pgfplotsinvokeforeach{1, ..., 3}{\addplot coordinates {(0,0)};}
    \addlegendentry{Interpolation data}
    \addlegendentry{\soloewnerk{}}
    \addlegendentry{\soloewnerd{}}
    
    \addlegendimage{empty legend}
    \addlegendentry{}
    \pgfplotsinvokeforeach{1, 2}{\addplot coordinates {(0,0)};}
    \addlegendentry{\soloewnerr{}}
    \addlegendentry{\loewner{}}
  \end{axis}
\end{tikzpicture}%
  \tikzexternaldisable%

  \caption{Artificial fishtail example: All methods provide a similar
    approximation accuracy.
    The model inferred by \soloewnerd{} performs insignificantly worse for
    frequencies between $1$ and $100$\,rad/s, while the model from
    \soloewnerr{} keeps the accuracy level also for higher frequencies.}
  \label{fig:fishtail}
\end{figure}
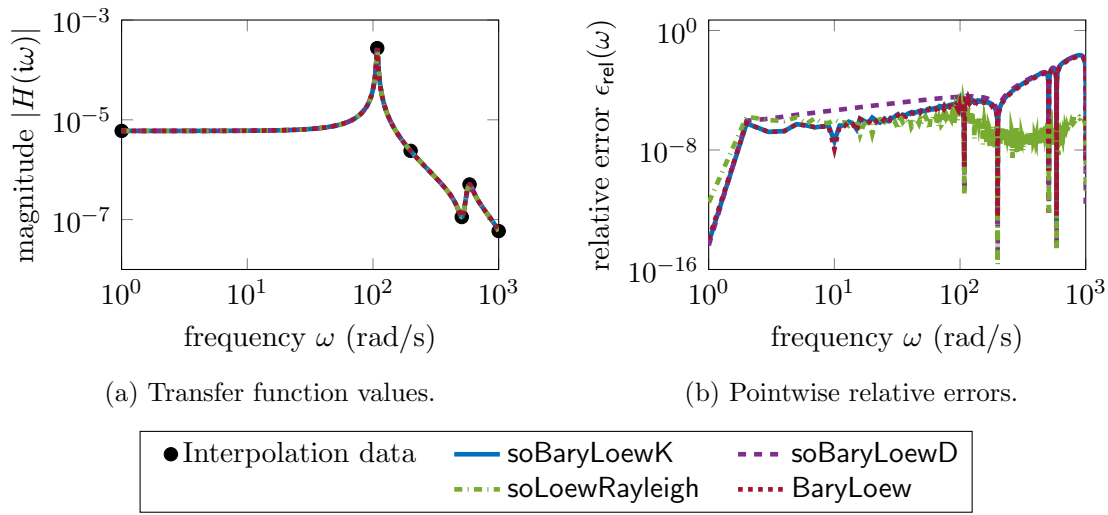

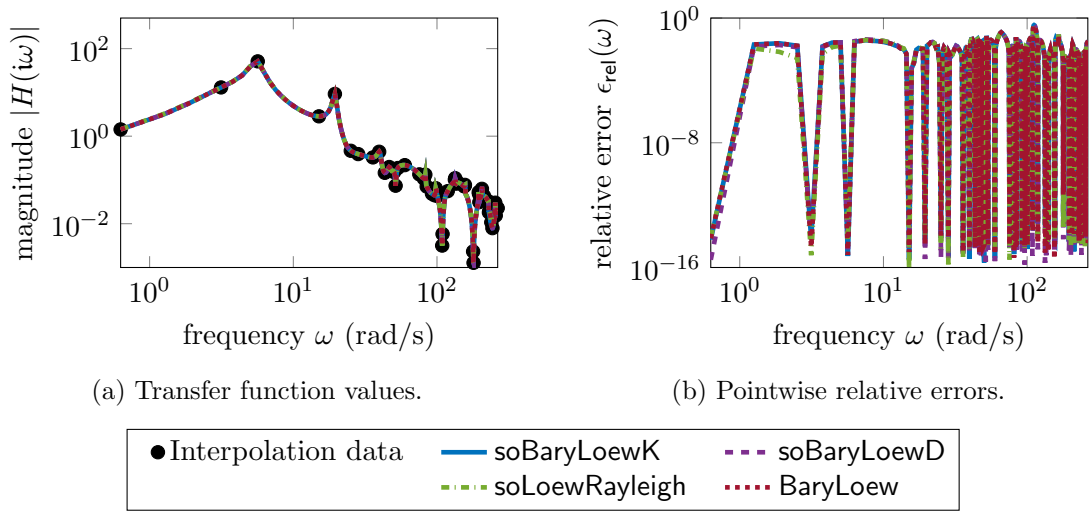
\begin{figure}[t]
  \centering
  \begin{subfigure}[b]{.49\textwidth}
    \centering
  \tikzexternalenable%
  \tikzsetnextfilename{aircraft_tf}%
  \begin{tikzpicture}
  \pgfplotstableread{graphics/data/aircraft_interpts.dat}\tablePTS
  \pgfplotstableread{graphics/data/aircraft_tf.dat}\tableTF
  
  \begin{loglogaxis}[%
    width  = .675\textwidth,
    height = .45\textwidth,
    scale only axis,
    xmin = 6.283185e-1,
    xmax = 2.645221e+2,
    ymin = 1e-3,
    ymax = 5e+2,
    xminorticks = false,
    yminorticks = false,
    xlabel = {frequency $\omega$ (rad/s)},
    ylabel = {magnitude $\lvert H(\i \omega) \rvert$},
    ylabel style   = {yshift = -.3em},
    scaled x ticks = false,
    x tick label style = {/pgf/number format/fixed},
    cycle list name    = plotlist,
    clip mode          = individual]
    
    \pgfplotsset{cycle list shift = 0}
    \addplot table[x index = 0, y index = 1] {\tablePTS};
    
    \pgfplotsset{cycle list shift = -1}
    \addplot table[x index = 2, y index = 3] {\tablePTS};
    
    \foreach \y in {2, 3, 4, 5} {
      \addplot table[x index = 0, y index = \y] {\tableTF};
    }
  \end{loglogaxis}
\end{tikzpicture}%
  \tikzexternaldisable%

    \caption{Transfer function values.}
    \label{fig:aircraft_tf}
  \end{subfigure}%
  \hfill%
  \begin{subfigure}[b]{.49\textwidth}
    \centering
  \tikzexternalenable%
  \tikzsetnextfilename{aircraft_err}%
  \begin{tikzpicture}
  \pgfplotstableread{graphics/data/aircraft_relerr.dat}\tableERR
  
  \begin{loglogaxis}[%
    width  = .675\textwidth,
    height = .45\textwidth,
    scale only axis,
    xmin = 6.283185e-1,
    xmax = 2.645221e+2,
    ymin = 1e-16,
    ymax = 1e-0,
    xminorticks = false,
    yminorticks = false,
    xlabel = {frequency $\omega$ (rad/s)},
    ylabel = {relative error $\relerr(\omega)$},
    ylabel style   = {yshift = -.3em},
    scaled x ticks = false,
    x tick label style = {/pgf/number format/fixed},
    cycle list name    = plotlist,
    clip mode          = individual]
    
    \pgfplotsset{cycle list shift = 1}
    \foreach \y in {1, 2, 3, 4} {
      \addplot table[x index = 0, y index = \y] {\tableERR};
    }
  \end{loglogaxis}
\end{tikzpicture}%
  \tikzexternaldisable%

    \caption{Pointwise relative errors.}
    \label{fig:aircraft_err}
  \end{subfigure}
  
  \vspace{.5\baselineskip}
  \tikzexternalenable%
  \tikzsetnextfilename{aircraft_legend}%
  \begin{tikzpicture}
  \begin{axis}[%
    hide axis,
    width  = 1mm,
    height = 1mm,
    scale only axis,
    xmin = 0,
    xmax = 1,
    ymin = 0,
    ymax = 1,
    legend columns = 3, 
    legend style   = {
      at     = {(0,0)},
      anchor = center,
      /tikz/every even column/.append style = {column sep = 0.4cm}},
    legend cell align  = {left},
    cycle list name    = plotlist,
    clip mode          = individual]
    
    \pgfplotsinvokeforeach{1, ..., 3}{\addplot coordinates {(0,0)};}
    \addlegendentry{Interpolation data}
    \addlegendentry{\soloewnerk{}}
    \addlegendentry{\soloewnerd{}}
    
    \addlegendimage{empty legend}
    \addlegendentry{}
    \pgfplotsinvokeforeach{1, 2}{\addplot coordinates {(0,0)};}
    \addlegendentry{\soloewnerr{}}
    \addlegendentry{\loewner{}}
  \end{axis}
\end{tikzpicture}%
  \tikzexternaldisable%

  \caption{Flexible aircraft example: All methods construct reduced-order models
    that recover the given data set with sufficient accuracy.
    For higher frequencies, more interpolation data is used due to the presence
    of many local maxima and minima.}
  \label{fig:aircraft}
\end{figure}

We first consider the case of models with energy dissipation, which need the
presence of a damping matrix $\bhD$.
Three examples are considered to test the proposed methods.
The butterfly gyroscope models the behavior of a micro-mechanical vibrating
gyroscope structure for the use in inertia-based navigation
systems~\cite{Bil05, morwiki_gyro}.
The artificial fishtail models the deformation of a silicon structure in the
shape of a fishtail used in the construction of underwater vehicles with
fish-like locomotion~\cite{SieKM18, SieKMetal19}.
Lastly, we have sampled data from a high-fidelity simulation of a flexible
aircraft model used in civil aeronautics to optimize lightweight
structures~\cite{PouQV18, morwiki_aircraft}.
The dimensions of the sampled models and the dimension of the constructed
second-order models are shown in \Cref{tab:dim}.
Note that we consider here SISO versions of these examples, which are originally
single-input/multi-output.

The results computed by the different methods are shown in
\Cref{fig:butterfly,fig:fishtail,fig:aircraft}, where we have set the
support points as
\begin{equation*}
  \sigma_{1} = \ldots = \sigma_{r} = \theta_{1} = \ldots = \theta_{r}
  = -(5 + 10^{-3}\i) \cdot \omega_{\max}
\end{equation*}
for the butterfly gyroscope and flexible aircraft, and
\begin{equation*}
  \sigma_{1} = \ldots = \sigma_{r} = \theta_{1} = \ldots = \theta_{r}
  = -(5 + 10^{-5}\i) \cdot \omega_{\max}
\end{equation*}
in the case of the fishtail example, where $\omega_{\max} \in \R$ is the upper
limit of the considered frequency interval.
The figures show the transfer function magnitudes of the constructed models
with the used interpolation data and the pointwise relative errors computed
with respected to all given data samples.
In the case of the butterfly and fishtail examples, these are $1\,000$
samples in the frequency range of interest and $421$ samples for the flexible
aircraft example.
For all three examples, the considered methods perform similarly well in terms
of the pointwise relative errors shown in 
\Cref{fig:butterfly,fig:fishtail,fig:aircraft}.
However, we can note that the learned models assume different spectral
properties.
In the case of the butterfly gyroscope, the proposed methods \soloewnerk{} and
\soloewnerd{} produce asymptotically stable reduced-order models due to the
choice of support points, in contrast to the classical \loewner{}, which has
one unstable pole.
The Rayleigh-damped approach \soloewnerr{} gives one unstable and one
infinite eigenvalue, where the infinite one likely results from the finite
arithmetic in the eigenvalue computations and the highly ill-conditioned
learned system matrices.
For the fishtail example, all computed reduced-order models are stable
and for the aircraft example, no reduced-order model is stable.
In particular, \soloewnerd{}, \soloewnerr{} and \loewner{} have three pairs of
unstable complex conjugate eigenvalues, while \soloewnerk{} has only two pairs.

%%%%%%%%%%%%%%%%%%%%%%%%%%%%%%%%%%%%%%%%%%%%%%%%%%%%%%%%%%%%%%%%%%%%%%%%%%%%%%%%

\subsection{Examples with zero damping matrix}

\begin{figure}[t]
  \centering
  \begin{subfigure}[b]{.49\textwidth}
    \centering
  \tikzexternalenable%
  \tikzsetnextfilename{bonemodel_tf}%
  \begin{tikzpicture}
  \pgfplotstableread{graphics/data/bonemodel_interpts.dat}\tablePTS
  \pgfplotstableread{graphics/data/bonemodel_tf.dat}\tableTF
  
  \begin{loglogaxis}[%
    width  = .675\textwidth,
    height = .45\textwidth,
    scale only axis,
    xmin = 6.283185e+0,
    xmax = 6.283185e+2,
    ymin = 1e-6,
    ymax = 1e+0,
    xminorticks = false,
    yminorticks = false,
    xlabel = {frequency $\omega$ (rad/s)},
    ylabel = {magnitude $\lvert H(\i \omega) \rvert$},
    ylabel style   = {yshift = -.3em},
    scaled x ticks = false,
    x tick label style = {/pgf/number format/fixed},
    cycle list name    = plotlist2,
    clip mode          = individual]
    
    \pgfplotsset{cycle list shift = 0}
    \addplot table[x index = 0, y index = 1] {\tablePTS};
    
    \pgfplotsset{cycle list shift = -1}
    \addplot table[x index = 2, y index = 3] {\tablePTS};
    
    \foreach \y in {2, 3, 4} {
      \addplot table[x index = 0, y index = \y] {\tableTF};
    }
  \end{loglogaxis}
\end{tikzpicture}%
  \tikzexternaldisable%

    \caption{Transfer function values.}
    \label{fig:bonemodel_tf}
  \end{subfigure}%
  \hfill%
  \begin{subfigure}[b]{.49\textwidth}
    \centering
  \tikzexternalenable%
  \tikzsetnextfilename{bonemodel_err}%
  \begin{tikzpicture}
  \pgfplotstableread{graphics/data/bonemodel_relerr.dat}\tableERR
  
  \begin{loglogaxis}[%
    width  = .675\textwidth,
    height = .45\textwidth,
    scale only axis,
    xmin = 6.283185e+0,
    xmax = 6.283185e+2,
    ymin = 1e-16,
    ymax = 1e-1,
    xminorticks = false,
    yminorticks = false,
    xlabel = {frequency $\omega$ (rad/s)},
    ylabel = {relative error $\relerr(\omega)$},
    ylabel style   = {yshift = -.3em},
    scaled x ticks = false,
    x tick label style = {/pgf/number format/fixed},
    cycle list name    = plotlist2,
    clip mode          = individual]
    
    \pgfplotsset{cycle list shift = 1}
    \foreach \y in {1, 2, 3} {
      \addplot table[x index = 0, y index = \y] {\tableERR};
    }
  \end{loglogaxis}
\end{tikzpicture}%
  \tikzexternaldisable%

    \caption{Pointwise relative errors.}
    \label{fig:bonemodel_err}
  \end{subfigure}
  
  \vspace{.5\baselineskip}
  \tikzexternalenable%
  \tikzsetnextfilename{bonemodel_legend}%
  \begin{tikzpicture}
  \begin{axis}[%
    hide axis,
    width  = 1mm,
    height = 1mm,
    scale only axis,
    xmin = 0,
    xmax = 1,
    ymin = 0,
    ymax = 1,
    legend columns = 4, 
    legend style   = {
      at     = {(0,0)},
      anchor = center,
      /tikz/every even column/.append style = {column sep = 0.4cm}},
    legend cell align  = {left},
    cycle list name    = plotlist2,
    clip mode          = individual]
    
    \pgfplotsinvokeforeach{1, ..., 4}{\addplot coordinates {(0,0)};}
    \addlegendentry{Interpolation data}
    \addlegendentry{\soloewnerkdo{}}
    \addlegendentry{\soloewnerr{}}
    \addlegendentry{\loewner{}}
  \end{axis}
\end{tikzpicture}%
  \tikzexternaldisable%

  \caption{Bone model example: For lower frequencies, the second-order methods
    produce models with at least one order of magnitude smaller relative errors
    than the classical Loewner framework.
    The curves of \soloewnerkdo{} and \soloewnerr{} are identical up to
    numerical round-off errors.}
  \label{fig:bonemodel}
\end{figure}
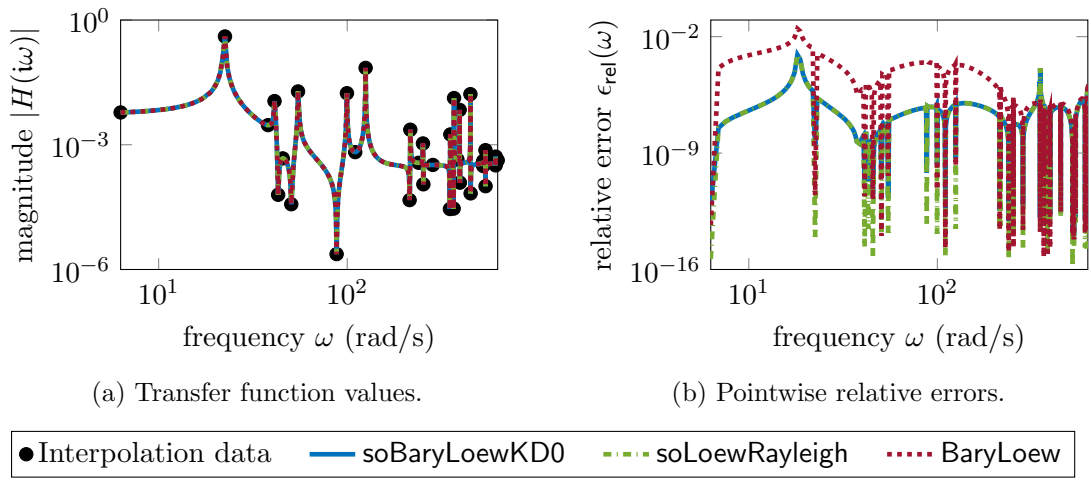

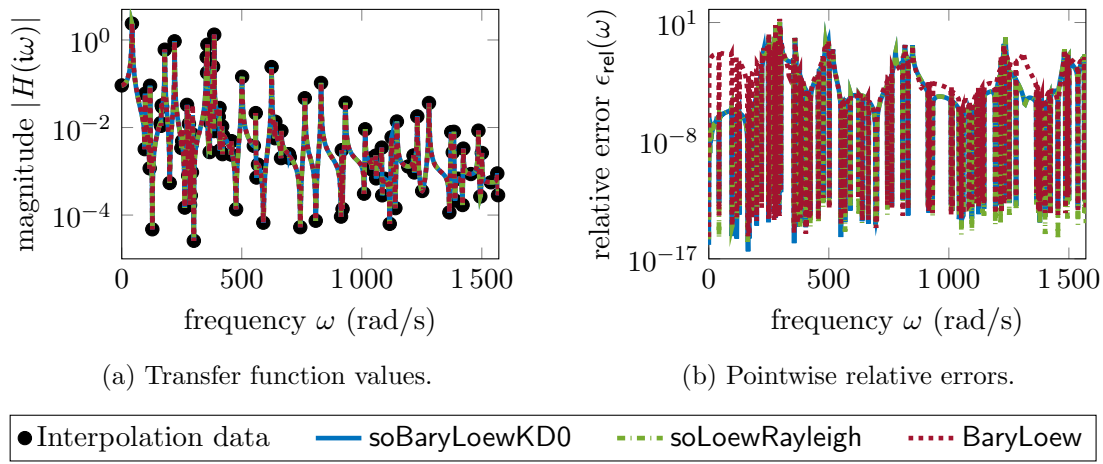
\begin{figure}[t]
  \centering
  \begin{subfigure}[b]{.49\textwidth}
    \centering
  \tikzexternalenable%
  \tikzsetnextfilename{plate_tf}%
  \begin{tikzpicture}
  \pgfplotstableread{graphics/data/plate_interpts.dat}\tablePTS
  \pgfplotstableread{graphics/data/plate_tf.dat}\tableTF
  
  \begin{semilogyaxis}[%
    width  = .675\textwidth,
    height = .45\textwidth,
    scale only axis,
    xmin = 0,
    xmax = 1.570796e+03,
    ymin = 1e-5,
    ymax = 5e+0,
    xminorticks = false,
    yminorticks = false,
    xlabel = {frequency $\omega$ (rad/s)},
    ylabel = {magnitude $\lvert H(\i \omega) \rvert$},
    ylabel style   = {yshift = -.3em},
    scaled x ticks = false,
    x tick label style = {/pgf/number format/fixed},
    cycle list name    = plotlist2,
    clip mode          = individual]
    
    \pgfplotsset{cycle list shift = 0}
    \addplot table[x index = 0, y index = 1] {\tablePTS};
    
    \pgfplotsset{cycle list shift = -1}
    \addplot table[x index = 2, y index = 3] {\tablePTS};
    
    \foreach \y in {2, 3, 4} {
      \addplot table[x index = 0, y index = \y] {\tableTF};
    }
  \end{semilogyaxis}
\end{tikzpicture}%
  \tikzexternaldisable%

    \caption{Transfer function values.}
    \label{fig:plate_tf}
  \end{subfigure}%
  \hfill%
  \begin{subfigure}[b]{.49\textwidth}
    \centering
  \tikzexternalenable%
  \tikzsetnextfilename{plate_err}%
  \begin{tikzpicture}
  \pgfplotstableread{graphics/data/plate_relerr.dat}\tableERR
  
  \begin{semilogyaxis}[%
    width  = .675\textwidth,
    height = .45\textwidth,
    scale only axis,
    xmin = 0,
    xmax = 1.570796e+03,
    ymin = 1e-17,
    ymax = 1e+2,
    xminorticks = false,
    yminorticks = false,
    xlabel = {frequency $\omega$ (rad/s)},
    ylabel = {relative error $\relerr(\omega)$},
    ylabel style   = {yshift = -.3em},
    scaled x ticks = false,
    x tick label style = {/pgf/number format/fixed},
    cycle list name    = plotlist2,
    clip mode          = individual]
    
    \pgfplotsset{cycle list shift = 1}
    \foreach \y in {1, 2, 3} {
      \addplot table[x index = 0, y index = \y] {\tableERR};
    }
  \end{semilogyaxis}
\end{tikzpicture}%
  \tikzexternaldisable%

    \caption{Pointwise relative errors.}
    \label{fig:plate_err}
  \end{subfigure}
  
  \vspace{.5\baselineskip}
  \tikzexternalenable%
  \tikzsetnextfilename{plate_legend}%
  \begin{tikzpicture}
  \begin{axis}[%
    hide axis,
    width  = 1mm,
    height = 1mm,
    scale only axis,
    xmin = 0,
    xmax = 1,
    ymin = 0,
    ymax = 1,
    legend columns = 4, 
    legend style   = {
      at     = {(0,0)},
      anchor = center,
      /tikz/every even column/.append style = {column sep = 0.5cm}},
    legend cell align  = {left},
    cycle list name    = plotlist2,
    clip mode          = individual]
    
    \pgfplotsinvokeforeach{1, ..., 4}{\addplot coordinates {(0,0)};}
    \addlegendentry{Interpolation data}
    \addlegendentry{\soloewnerkdo{}}
    \addlegendentry{\soloewnerr{}}
    \addlegendentry{\loewner{}}
  \end{axis}
\end{tikzpicture}%
  \tikzexternaldisable%

  \caption{Hysteretic plate example: For up to $50$\,rad/s, the second-order
    methods \soloewnerkdo{} and \soloewnerr{} produce relative errors that are
    at least four orders of magnitude smaller than the classical \loewner{}.
    The curves of \soloewnerkdo{} and \soloewnerr{} are identical up to
    numerical round-off.}
  \label{fig:plate}
\end{figure}

Now, we consider two examples with zero damping matrix, in order to test
\soloewnerkdo{}.
First, we have the bone model as a conservative system, which is used to
simulate the porous bone micro-architecture in studies of bone tissue under
different loads~\cite{VanWHetal95, morwiki_bone}.
As a second example, we consider the model of a vibrating plate that is equipped
with tuned vibration absorbers, which lead to hysteretic structural
damping~\cite{AumW23, supAumW22a}.
The results of the different methods can be seen in
\Cref{fig:bonemodel,fig:plate}.
In both examples, the second-order methods \soloewnerkdo{} and \soloewnerr{}
perform better in terms of the pointwise relative errors than the classical
\loewner{}.
This comes from the additional preservation of the damping model in these
methods.
In particular, we can observe that in the absence of any type of energy
dissipation, the two methods that preserve the zero damping structure outperform
the classical Loewner framework by several orders of magnitude.

Additionally, we note that the curves of \soloewnerkdo{} and
\soloewnerr{} are in fact identical in both examples.
This is a numerical verification that the barycentric form in
\Cref{cor:barytfKD0} describes exactly the same system that is recovered by
\soloewnerr{} with zero Rayleigh damping parameters, i.e., both methods
construct different realizations of exactly the same interpolatory second-order
systems.

%%%%%%%%%%%%%%%%%%%%%%%%%%%%%%%%%%%%%%%%%%%%%%%%%%%%%%%%%%%%%%%%%%%%%%%%%%%%%%%%
% CONCLUSIONS.                                                                 %
%%%%%%%%%%%%%%%%%%%%%%%%%%%%%%%%%%%%%%%%%%%%%%%%%%%%%%%%%%%%%%%%%%%%%%%%%%%%%%%%
  
\section{Conclusions}%
\label{sec:conclusions}

We have developed new structured barycentric forms for the transfer functions
of second-order systems for data-driven, interpolatory reduced-order
modeling.
Based on these barycentric forms, we have proposed three Loewner-like
algorithms for the construction of second-order systems from data.
Numerical experiments compared these new methods to the classical, unstructured
Loewner approach as well as to another Loewner-like method for the construction
of second-order systems from frequency domain data.
In all examples, the new structured approaches were able to provide a similar ,
and in some cases significantly better, approximation accuracy as the
established methods from the literature some of which do not obey to preserve
the structure.
Since the proposed algorithms rely on some fixed parameter choices to simplify
computations, we expect that including these additional ``support points'' as
parameters would significantly increase the approximation capabilities
of methods based on the presented structured barycentric forms.
However, we leave these considerations for future work.
Additionally, we have observed that these support points can be used to alter
the properties of the constructed system matrices, allowing, for example, the
construction of asymptotically stable second-order systems.

At the heart of this work are the new structured barycentric forms that allow
for a large bandwidth of new algorithms for learning structured models from
frequency domain data.
For the clarity of the presentation, we restricted the analysis in this work to
a purely interpolatory framework.
However, the use of the free parameters in the barycentric forms for
least-squares fitting will allow the derivation of methods like vector
fitting~\cite{GusS99} and AAA~\cite{NakST18} for second-order systems.
In particular, the presence of more parameters than in the unstructured,
first-order system case gives rise to a lot more variety in resulting
algorithms.
We will consider these ideas in future work.

%%%%%%%%%%%%%%%%%%%%%%%%%%%%%%%%%%%%%%%%%%%%%%%%%%%%%%%%%%%%%%%%%%%%%%%%%%%%%%%%
% *** ACKNOWLEDGEMENTS ***                                                     %
%%%%%%%%%%%%%%%%%%%%%%%%%%%%%%%%%%%%%%%%%%%%%%%%%%%%%%%%%%%%%%%%%%%%%%%%%%%%%%%%

\section*{Acknowledgments}%
\addcontentsline{toc}{section}{Acknowledgments}

The work of Gugercin is based upon work supported by the National Science
Foundation under Grant No. AMPS-1923221.

%%%%%%%%%%%%%%%%%%%%%%%%%%%%%%%%%%%%%%%%%%%%%%%%%%%%%%%%%%%%%%%%%%%%%%%%%%%%%%%%
% *** REFERENCES ***                                                           %
%%%%%%%%%%%%%%%%%%%%%%%%%%%%%%%%%%%%%%%%%%%%%%%%%%%%%%%%%%%%%%%%%%%%%%%%%%%%%%%%

\addcontentsline{toc}{section}{References}
\bibliographystyle{plainurl}
\bibliography{bibtex/myref}
  
\end{document}